\documentclass[final,onefignum,onetabnum]{siamart220329}

\usepackage{setspace}
\usepackage{lscape}
\usepackage{pdflscape}

\usepackage{amsmath,amssymb} 
\allowdisplaybreaks

\usepackage{wrapfig} 

\usepackage{mathrsfs} 
\usepackage{url} 
\usepackage{latexsym} 
\usepackage{bm} 

\usepackage{mathtools}

\usepackage{booktabs,multirow}
\usepackage{array}
\newcolumntype{P}[1]{>{\centering\arraybackslash}p{#1}}
\usepackage{diagbox}

\usepackage[noend]{algpseudocode}

\algnewcommand{\Break}{\textbf{break}}
\algnewcommand{\algorithmicgoto}{\textbf{go to} Line}
\algnewcommand{\Goto}[1]{\algorithmicgoto~\ref{#1}}
\algblockdefx{MRepeat}{EndRepeat}{\textbf{repeat}}{}
\algnotext{EndRepeat}

\usepackage{thmtools,thm-restate}

\expandafter\def\csname ver@etex.sty\endcsname{3000/12/31} 
\usepackage{autonum} 
\makeatletter
\autonum@generatePatchedReferenceCSL{Cref} 
\makeatother

\usepackage{multicol}
\usepackage{wasysym} 
\usepackage{adjustbox}
\usepackage{xparse} 
\usepackage{pbox} 
\usepackage{xspace} 

\usepackage{enumitem} 

\DeclareMathOperator{\EE}{\mathbb{E}}

\DeclarePairedDelimiterX\inner[2]{\langle}{\rangle}{{#1},{#2}}
\DeclarePairedDelimiter\abs{|}{|}
\DeclarePairedDelimiter\norm{\|}{\|}

\DeclarePairedDelimiter\set{\{}{\}}
\DeclarePairedDelimiter\prn{(}{)}
\DeclarePairedDelimiter\bra{[}{]}
\DeclarePairedDelimiterX\Set[2]{\{}{\}}{\mspace{2mu}{#1}\;\delimsize|\;{#2}\mspace{2mu}}
\DeclarePairedDelimiterX\Prn[2]{(}{)}{\mspace{2mu}{#1}\;\delimsize|\;{#2}\mspace{2mu}}
\DeclarePairedDelimiterX\Bra[2]{[}{]}{\mspace{2mu}{#1}\;\delimsize|\;{#2}\mspace{2mu}}

\newcommand{\R}{\mathbb R}

\renewcommand{\epsilon}{\varepsilon}

\NewDocumentCommand{\exsub}{s m O{} m}{%
  \IfBooleanT{#1}{\EE_{#2}\nolimits\bra*{#4}}%
  \IfBooleanF{#1}{\EE_{#2}\nolimits\bra[#3]{#4}}%
}
\NewDocumentCommand{\ex}{s O{} m}{%
  \IfBooleanT{#1}{\EE\nolimits\bra*{#3}}%
  \IfBooleanF{#1}{\EE\nolimits\bra[#2]{#3}}%
}
\NewDocumentCommand{\cex}{s O{} m m}{%
  \IfBooleanT{#1}{\EE\nolimits\Bra*{#3}{#4}}%
  \IfBooleanF{#1}{\EE\nolimits\Bra[#2]{#3}{#4}}%
}

\usepackage{CJKutf8}

\newcommand{\mathInd}{\hphantom{{}={}}}
\newcommand{\by}[2][]{\text{\pbox[c]{\textwidth}{(by \pbox[t]{\textwidth}{\,\!#2)#1}}}}

\crefname{algorithm}{Algorithm}{Algorithms}
\crefname{line}{Line}{Lines}
\crefname{section}{Section}{Sections}
\crefname{subsection}{Section}{Sections}
\crefname{appendix}{Appendix}{Appendices}
\crefname{table}{Table}{Tables}
\crefname{figure}{Figure}{Figures}
\crefname{equation}{}{}
\Crefname{equation}{Eq.}{Eqs.}

\setlist[itemize]{
  topsep=0.4\baselineskip,
  itemsep=0\baselineskip,
  leftmargin=1.5em,
}

\setlist[enumerate]{
  font=\upshape,
  label=(\alph*),
  ref=(\alph*),
  topsep=0.4\baselineskip,
  itemsep=0\baselineskip,
  leftmargin=2em,
}

\newlist{enuminasm}{enumerate}{1} 
\setlist[enuminasm]{
  font=\upshape,
  label=(\alph*),
  ref=\theassumption(\alph*),
  topsep=0.4\baselineskip,
  itemsep=0\baselineskip,
  leftmargin=2em,
}
\crefalias{enuminasmi}{assumption}

\newlist{enuminthm}{enumerate}{1}
\setlist[enuminthm]{
  font=\upshape,
  label=(\alph*),
  ref=\thetheorem(\alph*),
  topsep=0.4\baselineskip,
  itemsep=0\baselineskip,
  leftmargin=2em,
}
\crefalias{enuminthmi}{theorem}

\newlist{enuminlem}{enumerate}{1}
\setlist[enuminlem]{
  font=\upshape,
  label=(\alph*),
  ref=\thelemma(\alph*),
  topsep=0.4\baselineskip,
  itemsep=0\baselineskip,
  leftmargin=2em,
}
\crefalias{enuminlemi}{lemma}

\newlist{enuminprop}{enumerate}{1}
\setlist[enuminprop]{
  font=\upshape,
  label=(\alph*),
  ref=\theproposition(\alph*),
  topsep=0.4\baselineskip,
  itemsep=0\baselineskip,
  leftmargin=2em,
}
\crefalias{enuminpropi}{proposition}

\newlist{enumincond}{enumerate}{1}
\setlist[enumincond]{
  font=\upshape,
  label=(\alph*),
  ref=\thecondition(\alph*),
  topsep=0.4\baselineskip,
  itemsep=0\baselineskip,
  leftmargin=2em,
}
\crefalias{enumincondi}{condition}

\usepackage[caption=false]{subfig}
\usepackage{cite}
\usepackage{threeparttable}
\newsiamthm{assumption}{Assumption}
\crefname{assumption}{Assumption}{Assumptions}

\newcommand{\Ltrue}{L_f}
\newcommand{\Mtrue}{M_f}
\newcommand{\Lest}{L}
\newcommand{\Mest}{M}
\newcommand{\Lmax}{\bar L}
\newcommand{\Mmax}{\bar M}
\newcommand{\Linit}{\Lest_{\mathrm{init}}}

\newcommand{\Proposed}{\textsf{Proposed}\xspace}
\newcommand{\GD}{\textsf{GD}\xspace}
\newcommand{\JNJ}{\textsf{JNJ2018}\xspace}
\newcommand{\LL}{\textsf{LL2022}\xspace}
\newcommand{\OC}{\textsf{OC2015}\xspace}
\newcommand{\LBFGS}{\textsf{L-BFGS}\xspace}
\newcommand{\CG}{\textsf{CG}\xspace}

\usepackage{amsfonts}
\usepackage{graphicx}

\newsiamremark{remark}{Remark}
\newsiamremark{hypothesis}{Hypothesis}
\crefname{hypothesis}{Hypothesis}{Hypotheses}
\newsiamthm{claim}{Claim}

\headers{Parameter-free accelerated gradient descent}{N.~Marumo and A.~Takeda}

\title{Parameter-free accelerated gradient descent for nonconvex minimization\thanks{
  Submitted to the editors 12/13/2022.
\funding{This work was partially supported by JSPS KAKENHI (19H04069) and JST ERATO (JPMJER1903).}}}

\author{%
  Naoki Marumo%
  \thanks{Graduate School of Information Science and Technology, University of Tokyo, Tokyo, Japan (\email{marumo@mist.i.u-tokyo.ac.jp})}\and
  Akiko Takeda\footnotemark[2]
  \thanks{Center for Advanced Intelligence Project, RIKEN, Tokyo, Japan}%
}

\begin{document}

\maketitle

\begin{abstract}
  We propose a new first-order method for minimizing nonconvex functions with a Lipschitz continuous gradient and Hessian. The proposed method is an accelerated gradient descent with two restart mechanisms and finds a solution where the gradient norm is less than $\epsilon$ in $O(\epsilon^{-7/4})$ function and gradient evaluations. Unlike existing first-order methods with similar complexity bounds, our algorithm is parameter-free because it requires no prior knowledge of problem-dependent parameters, e.g., the Lipschitz constants and the target accuracy $\epsilon$. The main challenge in achieving this advantage is estimating the Lipschitz constant of the Hessian using only first-order information. To this end, we develop a new Hessian-free analysis based on two technical inequalities: a Jensen-type inequality for gradients and an error bound for the trapezoidal rule.
  Several numerical results illustrate 
  that the proposed method performs comparably to existing algorithms with similar complexity bounds, even without parameter tuning.
\end{abstract}



\section{Introduction}
This paper studies general nonconvex optimization problems:
\begin{align}
  \min_{x \in \R^d} \ 
  f(x),
\end{align}
where $f: \R^d \to \R$ is bounded below (i.e., $\inf_{x \in \R^d} f(x) > - \infty$) and satisfies the following Lipschitz assumption.
\begin{assumption}
  \label{asm:gradient_hessian_lip}
  Let $\Ltrue, \Mtrue > 0$ be constants.
  \begin{enuminasm}
    \item
    \label{asm:gradient_lip}
    $\norm*{\nabla f(x) - \nabla f(y)} \leq \Ltrue \norm*{x - y}$ for all $x, y \in \R^d$,
    \item
    \label{asm:hessian_lip}
    $\norm*{\nabla^2 f(x) - \nabla^2 f(y)} \leq \Mtrue \norm*{x - y}$ for all $x, y \in \R^d$.
  \end{enuminasm}
\end{assumption}
For such general problems, algorithms that access $f$ only through function and gradient evaluations, called \emph{first-order methods} \cite{beck2017first,lan2020first}, have attracted widespread attention because of their computational efficiency.
The classical result for nonconvex problems is that gradient descent finds an $\epsilon$-stationary point (i.e., $x \in \R^d$ where $\norm*{\nabla f(x)} \leq \epsilon$) in $O(\epsilon^{-2})$ function and gradient evaluations under \cref{asm:gradient_lip}.
This complexity bound cannot be improved only with \cref{asm:gradient_lip} \cite{cartis2010complexity,carmon2020lower}.

More sophisticated first-order methods \cite{carmon2017convex,xu2017neon+,allen2018neon2,jin2018accelerated,li2022restarted} have been studied to improve the complexity bound under an additional assumption, \cref{asm:hessian_lip}.%
\footnote{
  Note that although the assumption involves Hessians, they are only used in the analysis and not in the implementation of the above algorithms.
}
These methods can find an $\epsilon$-stationary point in only $O(\epsilon^{-7/4})$ or $\tilde O(\epsilon^{-7/4})$ function and gradient evaluations under \cref{asm:gradient_hessian_lip}.%
\footnote{
  The notation $\tilde O$ hides polylogarithmic factors in $\epsilon^{-1}$.
  For example, the method of \cite{jin2018accelerated} has a complexity bound of $O(\epsilon^{-7/4} (\log \epsilon^{-1})^6)$.
}
Although the methods are theoretically attractive, one faces challenges when applying them to real-world problems; they require prior knowledge of the Lipschitz constants $L_f$ and $M_f$ and the target accuracy $\epsilon$ to choose the algorithm's parameters.
For example, the existing accelerated gradient descent (AGD) methods \cite{jin2018accelerated,li2022restarted} for nonconvex problems set the step size dependently on $L_f$ and the acceleration parameter dependently on $L_f$, $M_f$, and $\epsilon$.
The values of $L_f$, $M_f$, and $\epsilon$ are unknown in many situations, which necessitates painstaking parameter tuning.

\begin{figure}[t]
  \centering
  \includegraphics[height=3.5ex]{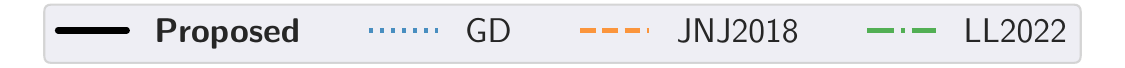}\par\smallskip%
  \includegraphics[width=0.33\linewidth]{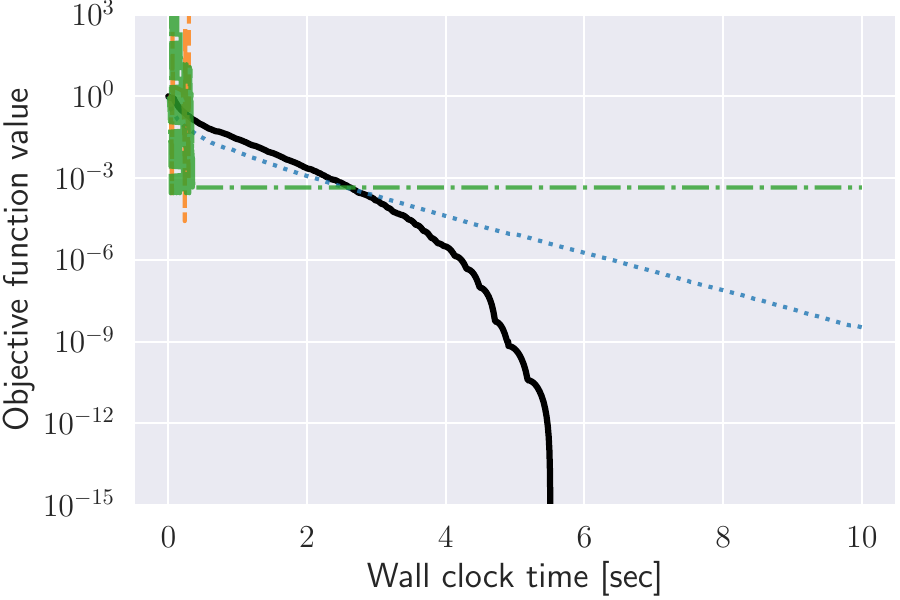}\hfill%
  \includegraphics[width=0.33\linewidth]{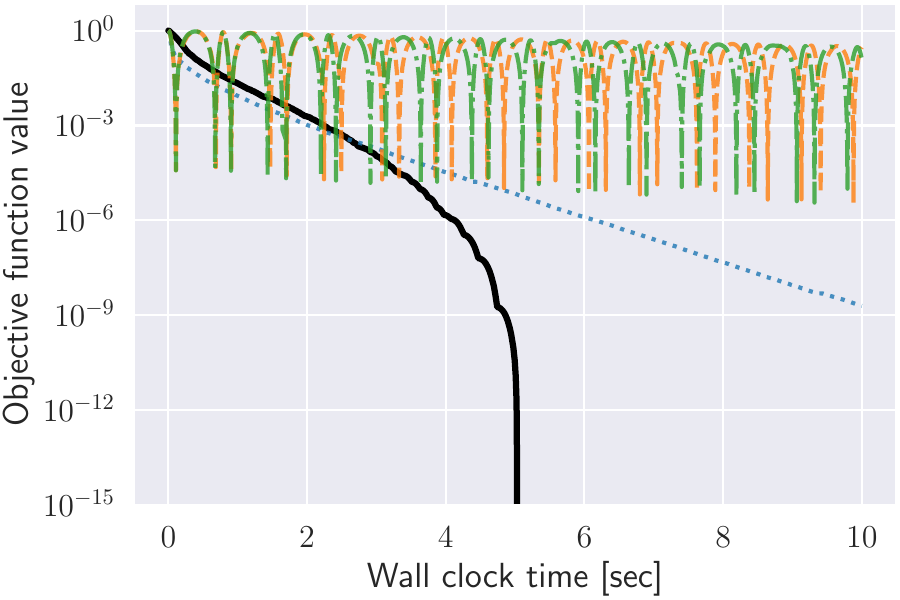}\hfill%
  \includegraphics[width=0.33\linewidth]{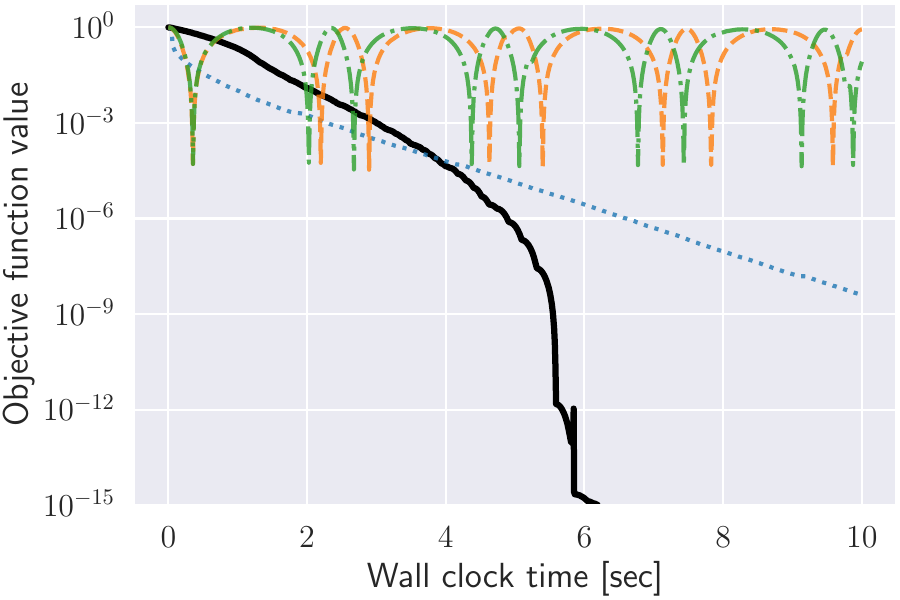}\par\medskip%
  \includegraphics[width=0.33\linewidth]{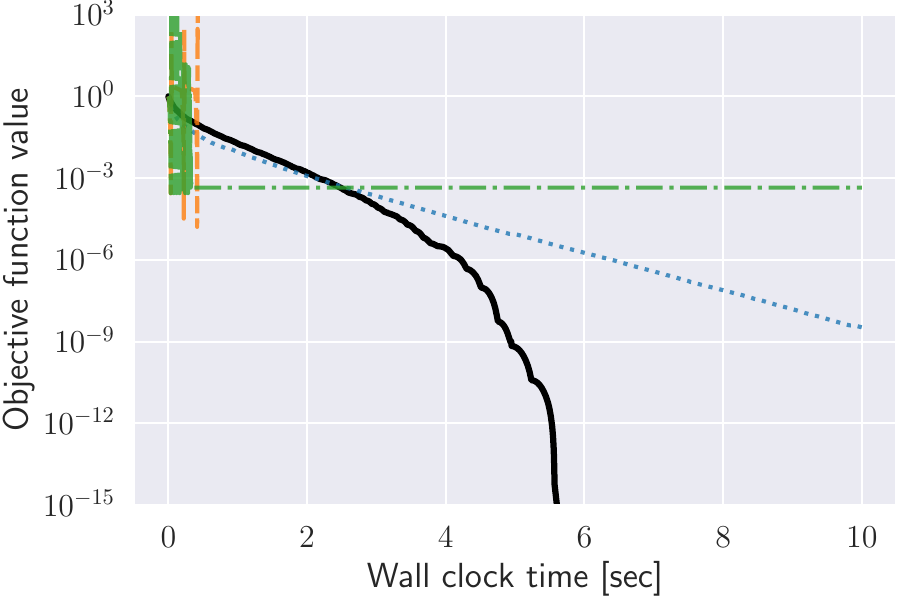}\hfill%
  \includegraphics[width=0.33\linewidth]{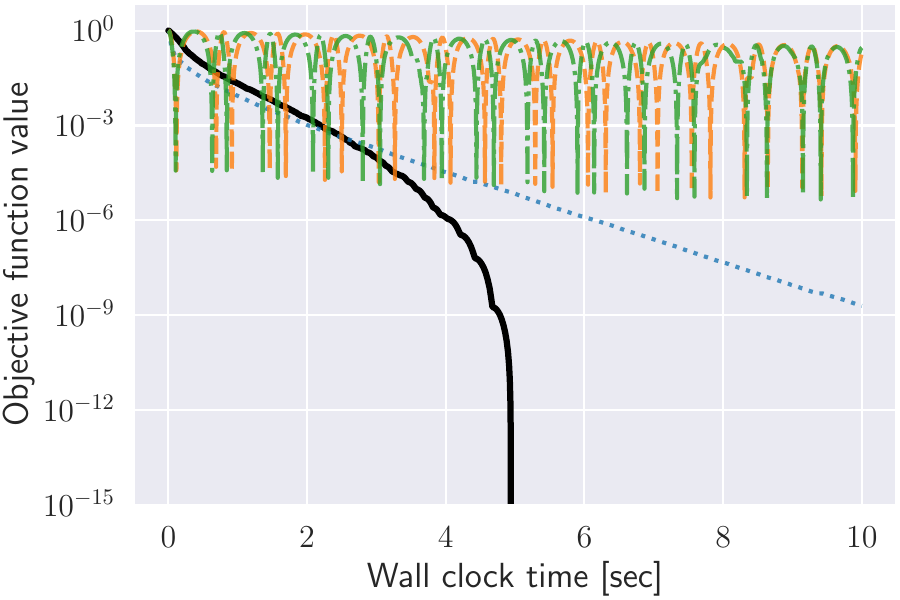}\hfill%
  \includegraphics[width=0.33\linewidth]{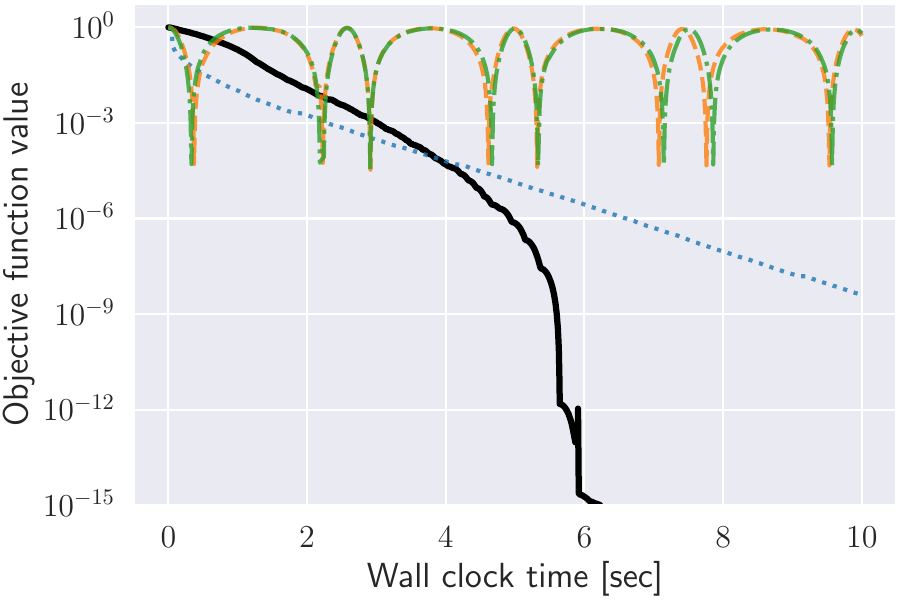}\par\medskip%
  \includegraphics[width=0.33\linewidth]{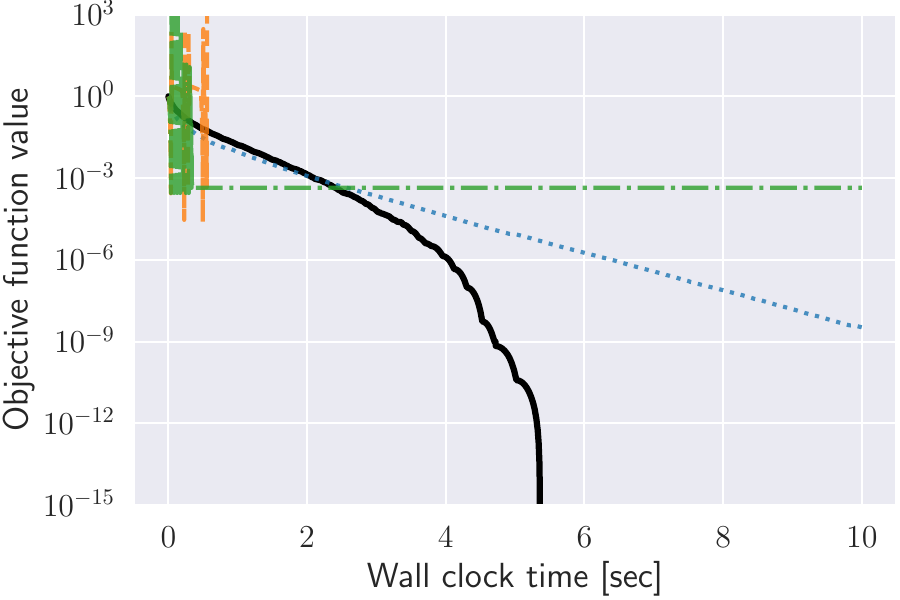}\hfill%
  \includegraphics[width=0.33\linewidth]{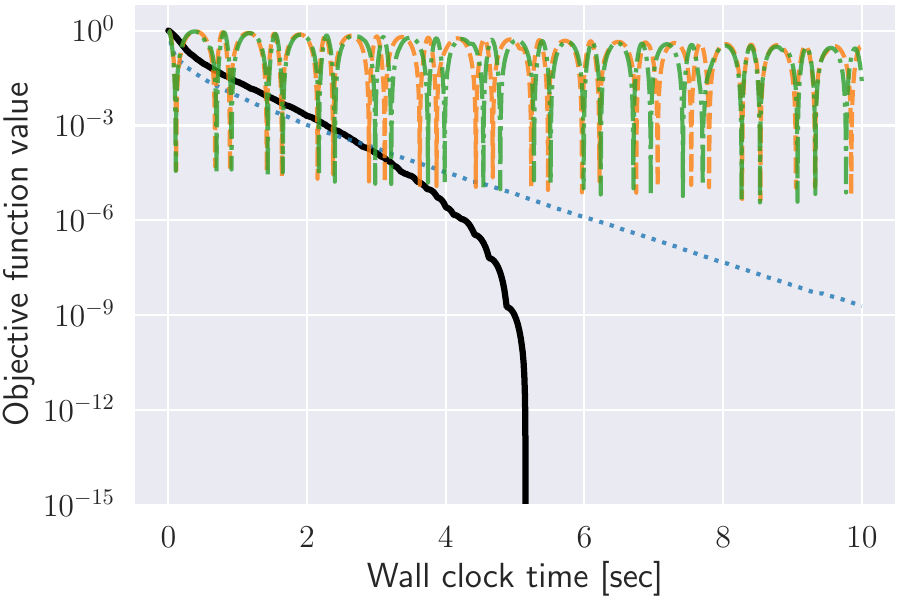}\hfill%
  \includegraphics[width=0.33\linewidth]{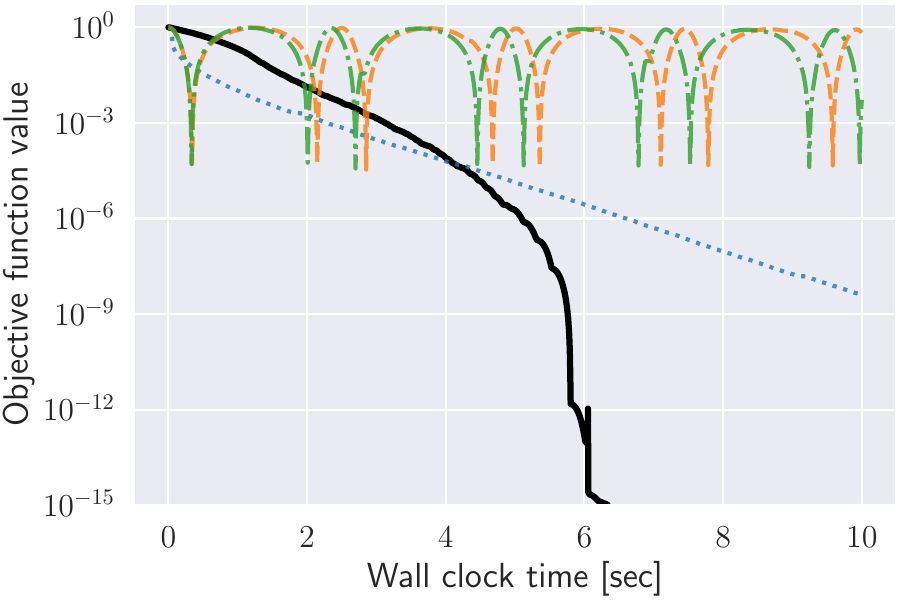}\par%
  \caption{%
    Minimization of the Rosenbrock function~\cite{rosenbrock1960automatic}: $\min_{(x, y) \in \R^2}\, (x - 1)^2 + 100 (y - x^2)^2$.
    \Proposed: \cref{alg:proposed_agd},
    \GD: gradient descent with Armijo-type backtracking,
    \JNJ: \cite[Algorithm~2]{jin2018accelerated},
    \LL: \cite[Algorithm~2]{li2022restarted}.
    All four methods require a guess of $L_f$, which is set to $10^2$, $10^3$, and $10^4$ (left to right).
    \Proposed, \JNJ, and \LL require a guess of $M_f$, which is set to $10^0$, $10^1$, and $10^2$ (top to bottom).
    \Proposed and \GD use the guesses as the initial estimate, while \JNJ and \LL use them as the true value.
  }
  \label{fig:rosenbrock_intro}
\end{figure}

To alleviate the problem, we propose a new first-order method.
The proposed method is an AGD equipped with two restart mechanisms and enjoys the following advantages.
\begin{itemize}
  \item
  Our method finds an $\epsilon$-stationary point in $O(\epsilon^{-7/4})$ function and gradient evaluations under \cref{asm:gradient_hessian_lip}.
  \item
  Our method estimates the Lipschitz constants $L_f$ and $M_f$ instead of requiring prior knowledge.
  It automatically updates the estimates, and their initial guesses are allowed to deviate from the actual Lipschitz constants, both in theory and practice.
  \item
  Our method also does not need the target accuracy $\epsilon$ as input, which implies its global convergence (\cref{cor:global_convergence}).
  \item
  We also conducted numerical experiments with several instances.
  The results show that the proposed method performs comparably to existing algorithms with similar complexity bounds, even without parameter tuning.
\end{itemize}
As shown in \cref{fig:rosenbrock_intro}, the proposed method successfully converged to a stationary point for a nonconvex test problem, regardless of the initial guesses of the Lipschitz constants, and it outperformed other state-of-the-art methods.

The main challenge in achieving the above advantages is estimating $M_f$ by using only first-order information, i.e., $f$ and $\nabla f$.
To this end, we establish a \emph{Hessian-free} analysis with two inequalities less familiar in the optimization context: a Jensen-type inequality for gradients and an error bound for the trapezoidal rule.
These two inequalities do not include the Hessian itself but rather its Lipschitz constant $M_f$, and they enable us to estimate $M_f$ and simplify the analysis.


\paragraph{Notation}
For vectors $a, b \in \R^d$, let $\inner{a}{b}$ denote the dot product and $\norm{a}$ denote the Euclidean norm.
For a matrix $A \in \R^{m \times n}$, let $\norm{A}$ denote the operator norm, or equivalently the largest singular value.

\section{Related work}
\label{sec:related_work}
This section reviews algorithms and techniques related to this work and characterizes our method in comparison with the existing methods.

\begin{table}[t]
  \centering
  \begin{threeparttable}
    \caption{
      Nonconvex optimization methods with complexity bounds of $O(\epsilon^{-7/4})$ or $\tilde O(\epsilon^{-7/4})$ under \cref{asm:gradient_hessian_lip}.
      The lower six methods are first-order methods that do not employ Hessian-vector products (HVP).
      The lower two methods have complexity bounds of $O(\epsilon^{-7/4})$, while the others have polylogarithmic factors, i.e., $\tilde O(\epsilon^{-7/4})$.
      The ``No knowledge'' columns indicate whether each algorithm requires prior knowledge of parameters to obtain the complexity bounds.
    }
    \label{table:existing_methods_and_ours}
    \small
    \def\arraystretch{1.2}
    \setlength{\tabcolsep}{0.5em}
    \begin{tabular}{@{}cccccccc@{}}\toprule
      \multirow{2}{*}{Reference}
      & \multirow{2}{*}{No HVP} & \multirow{2}{*}{\pbox[c]{12em}{Second-order\\stationary}} & \multirow{2}{*}{Deterministic} & \multirow{2}{*}{$O(\epsilon^{-7/4})$} & \multicolumn{3}{c}{No knowledge} \\\cmidrule(lr){6-8}
      \addlinespace[-0.7ex]                                                                                                                                                                                                               
      &                                                      &                                                                  &                                &                                       & $L_f$    & $M_f$    & $\epsilon$ \\\midrule
      \cite[Theorem~5.3]{carmon2018accelerated}
      &                         & \checked                                                         &                                &                                       &          &          &            \\
      \cite[Theorem~1]{agarwal2017finding}\tnote{a}
      &                         & \checked                                                         &                                &                                       &          &          &            \\
      \cite[Corollary~16]{royer2018complexity}
      &                         & \checked                                                         &                                &                                       & \checked & \checked &            \\
      \cite[Corollary~2]{royer2020newton}
      &                         & \checked                                                         &                                &                                       & \checked & \checked &            \\
      \cite[Corollary~1]{royer2020newton}
      &                         &                                                                  &  \checked                   &                                       & \checked & \checked &            \\\cmidrule{1-8}
      \cite[Theorem~2]{xu2017neon+}
      & \checked                & \checked                                                         &                                &                                       &          &          &            \\
      \cite[Theorem~4]{allen2018neon2}
      & \checked                & \checked                                                         &                                &                                       &          &          &            \\
      \cite[Theorem~8]{jin2018accelerated}
      & \checked                & \checked                                                         &                                &                                       &          &          &            \\
      \cite[Theorem~1]{carmon2017convex}
      & \checked                &                                                                  & \checked                       &                                       &          &          &            \\
      \cite[Theorem~2.2]{li2022restarted}
      & \checked                &                                                                  & \checked                       & \checked                              &          &          &            \\
      \textbf{This work}
      & \checked                &                                                                  & \checked                       & \checked                              & \checked & \checked & \checked   \\
      \bottomrule
    \end{tabular}
    \begin{tablenotes}
      \footnotesize
      \item[a] The conference paper \cite[Theorem~1]{agarwal2017finding} may seem to suggest that the result of \cite{agarwal2017finding} is deterministic. However, according to the full version (\url{https://arxiv.org/abs/1611.01146}), the proposed algorithm uses approximate PCA as a subroutine, which is a randomized algorithm, as stated in Theorem~2.5 in the full version. We thus treat \cite[Theorem~1]{agarwal2017finding} as a non-deterministic result.
    \end{tablenotes}
  \end{threeparttable}
\end{table}

\subsection{Methods with complexity of $O(\epsilon^{-7/4})$ or $\tilde O(\epsilon^{-7/4})$}
\cref{table:existing_methods_and_ours} summarizes the algorithms reviewed in this section.

Carmon et~al.~\cite{carmon2017convex} developed a first-order method that finds an $\epsilon$-stationary point in $\tilde O(\epsilon^{-7/4})$ function and gradient evaluations under \cref{asm:gradient_hessian_lip}.
Li and Lin~\cite{li2022restarted} further improved the complexity bound to $O(\epsilon^{-7/4})$ with a simpler algorithm.

There is another research stream aimed at finding a second-order $\epsilon$-stationary point, i.e., $x \in \R^d$ where $\norm*{\nabla f(x)} \leq \epsilon$ and $\nabla^2 f(x) \succeq - \sqrt{\epsilon} I$.
Some methods \cite{agarwal2017finding,carmon2018accelerated,royer2018complexity,royer2020newton} find second-order stationary points in $\tilde O(\epsilon^{-7/4})$ evaluations though they utilize Hessian-vector products and thus are not first-order methods.
Methods that achieve complexity bounds of $\tilde O(\epsilon^{-7/4})$ without Hessian-vector products also have been proposed \cite{xu2017neon+,allen2018neon2,jin2018accelerated}.
They can also be used to find a (first-order) stationary point but do not yield a complexity bound of $O(\epsilon^{-7/4})$ as in \cite{li2022restarted}.

Algorithms with such improved bounds can be complicated in that they call another iterative method to solve subproblems.
Some methods \cite{carmon2017convex,carmon2018accelerated,xu2017neon+,allen2018neon2} iteratively minimize $\ell_2$-regularized objectives by AGD for strongly convex functions, and other methods \cite{agarwal2017finding,royer2018complexity} solve other subproblems.
Jin et~al.~\cite{jin2018accelerated} obtained a simpler algorithm by removing the need for such a nested-loop structure.
The simplest one is arguably that of \cite{li2022restarted}; it is a classical AGD with a restart mechanism and does not use any other techniques of earlier algorithms, such as randomness, negative curvature descent, or minimizing surrogate functions.
Notably, the most straightforward algorithm achieves the best complexity bound, $O(\epsilon^{-7/4})$, for first-order stationary points.

All of the abovementioned methods assume a Lipschitz continuous gradient and Hessian.
Some of them \cite{carmon2018accelerated,agarwal2017finding,carmon2017convex,xu2017neon+,allen2018neon2,jin2018accelerated,li2022restarted} additionally assume that (estimates of) both Lipschitz constants $L_f$ and $M_f$ are given as input to achieve the desired complexity bounds.
If incorrect estimates are given, there is no guarantee of convergence, let alone complexity bounds, as shown in \cref{fig:rosenbrock_intro}.\footnote{
    With some modification, these methods~\cite{carmon2018accelerated,agarwal2017finding,carmon2017convex,xu2017neon+,allen2018neon2,jin2018accelerated,li2022restarted} might be improved to not require $L_f$ as input, for example, by incorporating the Armijo rule reviewed in \cref{sec:review_estimation_lipschitz}.
    However, it is non-trivial to modify them so as to use neither the actual value of $M_f$ nor Hessian-vector products, as discussed in \cref{sec:review_estimation_lipschitz,sec:preliminary}.
}
Other methods \cite{royer2018complexity,royer2020newton} relaxed this restriction; they do not require the value of $L_f$ and $M_f$ as input, though they require Hessian-vector products instead.
It should be noted that, unlike our method, their complexity has a worse dependence on $M_f$ than the methods \cite{carmon2018accelerated,agarwal2017finding,xu2017neon+,allen2018neon2,jin2018accelerated,carmon2017convex,li2022restarted} that require $L_f$ and $M_f$ as input.
Our first-order method is the first to achieve the complexity bound of $O(\epsilon^{-7/4})$ (or $\tilde O(\epsilon^{-7/4})$) without requiring Hessian-vector products or knowledge of the Lipschitz constants.

The existing methods require a target accuracy, $\epsilon$, as input; for example, AGD methods \cite{jin2018accelerated,li2022restarted} choose an acceleration parameter dependently on $\epsilon$.
In contrast, our algorithm does not require pre-determination of accuracy; in other words, when stopped at an arbitrary time, the algorithm is guaranteed to produce a solution with accuracy corresponding to the number of iterations up to that point.


\subsection{Estimation of Lipschitz constants}
\label{sec:review_estimation_lipschitz}
Estimating the Lipschitz constant $L_f$ of $\nabla f$ dates back at least to backtracking with the Armijo rule for gradient descent~\cite{armijo1966minimization}.
The method estimates $L_f$ on the basis of the following lemma and has a complexity bound of $O(\epsilon^{-2})$, the same as the case where $L_f$ is known.
\begin{lemma}
  \label{lem:eq:property_gradlip_obj}
  Under \cref{asm:gradient_lip}, the following holds for any $x, y \in \R^d$:
  \begin{align}
    f(x) - f(y)
    &\leq
    \inner*{\nabla f(y)}{x - y}
    + \frac{\Ltrue}{2} \norm*{x - y}^2.
  \end{align}
\end{lemma}
Because of their practical importance, first-order methods with Armijo-type or other estimation scheme have been thoroughly studied for both convex \cite{beck2009fast,scheinberg2014fast,calatroni2019backtracking} and nonconvex \cite{bonettini2016variable,boumal2018global,liang2021average} problems.
Our method also estimates $L_f$ by using backtracking, but with a different rule than the Armijo one to make it compatible with acceleration for nonconvex problems.

Estimating the Lipschitz constant $M_f$ of $\nabla^2 f$ has also been studied in convex \cite{jiang2020unified,carmon2022optimal,grapiglia2017regularized,cartis2012evaluation} and nonconvex \cite{cartis2011adaptive,cartis2011adaptive2,gould2012updating,royer2018complexity,grapiglia2017regularized,agarwal2021adaptive} settings.
However, these schemes employ Hessian matrices or Hessian-vector products and are beyond the scope of first-order methods.
Estimating $M_f$ only with first-order information has so far remained a challenging problem; here, we resolve it affirmatively.

\section{Two Hessian-free inequalities}
\label{sec:preliminary}
Estimating $M_f$ is nontrivial.
One might think of estimating it by backtracking with inequalities required for complexity analysis, like the Armijo rule.
However, a problem arises because the existing complexity analyses of first-order methods \cite{carmon2017convex,xu2017neon+,allen2018neon2,jin2018accelerated,li2022restarted} use inequalities involving the Hessian such as
\begin{align}
  &\norm{\nabla f(x) - \nabla f(y) - \nabla^2 f(y) (x - y)}
  \leq
  \dfrac{M_f}{2} \norm*{x - y}^2,\\
  &f(x) - f(y)
  \leq
  \inner{\nabla f(y)}{x - y} + \dfrac{1}{2} \inner*{\nabla^2 f(y) (x - y)}{x - y} + \dfrac{M_f}{6} \norm*{x - y}^3.
\end{align}
If Hessian-vector multiplication were permissible, the estimation of $M_f$ could be achieved through some backtracking-like technique with the above inequalities.
However, since we are here targeting a first-order method, we can only access function values and gradients and thus cannot estimate $M_f$ on the basis of existing analyses.
A new technique of complexity analysis should be developed to estimate $M_f$.

Instead of the above inequalities, we employ the following two Hessian-free inequalities to estimate $M_f$ and analyze the algorithm.
\begin{lemma}[Jensen-type inequality for gradients]
  \label{lem:gradient_jensen}
  Under \cref{asm:hessian_lip}, the following holds for any $z_1, \dots, z_n \in \R^d$ and $\lambda_1,\dots,\lambda_n \geq 0$ such that $\sum_{i=1}^n \lambda_i = 1$:
  \begin{align}
    \norm*{
      \nabla f \prn*{ \sum_{i=1}^n \lambda_i z_i }
      - \sum_{i=1}^n \lambda_i \nabla f \prn*{ z_i }
    }
    \leq
    \frac{\Mtrue}{2} \sum_{1 \leq i < j \leq n} \lambda_i \lambda_j \norm*{z_i - z_j}^2.
  \end{align}
\end{lemma}
\begin{lemma}[Error bound for trapezoidal rule]
  \label{lem:trapezoidal_rule_error}
  Under \cref{asm:hessian_lip}, the following holds for any $x, y \in \R^d$:
  \begin{align}
    f(x) - f(y)
    &\leq
    \frac{1}{2} \inner*{\nabla f(x) + \nabla f(y)}{x - y}
    + \frac{\Mtrue}{12} \norm*{x - y}^3.
  \end{align}
\end{lemma}
To make the paper self-contained, we present the proofs in \cref{sec:proof_lem_lip_hessian}.
These two inequalities enable us not only to estimate the Hessian's Lipschitz constant but also to make the analysis more straightforward (see the last paragraph of \cref{sec:discussion} for a detailed discussion).
Though these inequalities are not new, using them in complexity analyses of optimization methods is new, to the best of our knowledge.

\cref{lem:gradient_jensen} shows that if $z_1,\dots,z_n$ are close together, then the gradient at a weighted average $\sum_{i=1}^n \lambda_i z_i$ can be approximated by the average of the gradients at $z_1,\dots,z_n$.
This lemma is analogous to Jensen's inequality, $\phi(\sum_{i=1}^n \lambda_i z_i) - \sum_{i=1}^n \lambda_i \phi(z_i) \leq 0$ for any convex function $\phi: \R^d \to \R$.
In convex optimization, Jensen's inequality is often used to show that the objective value at an averaged solution is small \cite{duchi2012ergodic,zhang2013linear,davis2016convergence,rakhlin2012making,chambolle2016ergodic}.
In contrast, our nonconvex analysis will use \cref{lem:gradient_jensen} to prove that the gradient norm at an averaged solution is small.
An inequality similar to \cref{lem:gradient_jensen} can be found in \cite[Corollary~2]{zlobec2004jensen}.

\cref{lem:trapezoidal_rule_error} is similar to \cref{lem:eq:property_gradlip_obj}; a notable difference is that the last term of \cref{lem:trapezoidal_rule_error} is not $O(\norm{x - y}^2)$ but rather $O(\norm{x - y}^3)$, which is a significant improvement when $\norm{x - y}$ is small.
The one-dimensional version ($d = 1$) of \cref{lem:trapezoidal_rule_error} is known as an error bound for the trapezoidal rule in the numerical analysis community (e.g., \cite[Eq.~(5.1.4)]{atkinson2008introduction}).

\section{Algorithm}

The proposed method is the classical accelerated gradient descent (AGD) \cite{nesterov1983method,nesterov2018lectures} with two restart schemes, presented in \cref{alg:proposed_agd}.
In the algorithm, the iteration counter $k$ is reset to $0$ when AGD restarts on Line \ref{alg-line-agd:restart_unsuccessful} or \ref{alg-line-agd:restart_successful}, whereas the total iteration counter $K$ is not.
We refer to the period between a reset of $k$ and the next reset as an \emph{epoch}.
Note that the total iteration counter $K$ in the algorithm is unnecessary to implement; it is included here only to make the statements in our analysis concise.

The algorithm uses estimates $L$ and $M_k$ for the Lipschitz constants $L_f$ and $M_f$.
The estimate $L$ is fixed during an epoch, while $M_k$ is updated at each iteration, having the subscript $k$.
Now let us describe each operation of \cref{alg:proposed_agd}.

\paragraph{Solution update (Lines \ref{alg-line-agd:update_x} and \ref{alg-line-agd:update_y})}
With an estimate $L$ of the Lipschitz constant $L_f$, AGD defines solution sequences $(x_k)$ and $(y_k)$ as follows:
$y_0 = x_0$ and
\begin{align}
  \begin{aligned}
    x_k
    &=
    y_{k-1} - \frac{1}{L} \nabla f(y_{k-1}),\\
    y_k
    &=
    x_k + \theta_k (x_k - x_{k-1})
  \end{aligned}
  \label{eq:update_xy}
\end{align}
for $k \geq 1$, where $0 \leq \theta_k \leq 1$ is the acceleration parameter.
Let $x_{-1} \coloneqq x_0$ for convenience, which makes the lower equation of \cref{eq:update_xy} valid for $k = 0$.
Setting a suitable value of $\theta_k$ is crucial for deriving an improved complexity bound over that of gradient descent.
In particular, we propose to set
\begin{align}
  \theta_k = \frac{k}{k+1}
  \label{eq:def_theta_k}
\end{align}
for $k \geq 1$.
This choice significantly differs from $\theta_k = 1 - (M_f \epsilon)^{1/4} / (4 \sqrt{L_f})$ \cite{jin2018accelerated} and $\theta_k = 1 - 2 (M_f \epsilon)^{1/4} / \sqrt{L_f}$ \cite{li2022restarted} as were proposed for AGD under \cref{asm:gradient_hessian_lip}.
Our $\theta_k$ does not depend on the true Lipschitz constants $L_f$ and $M_f$ or their estimates.
This independence will help later when computing an estimate of $M_f$.

\begin{algorithm}[t]
  \caption{Proposed AGD with restart~\label{alg:proposed_agd}} 
  \begin{algorithmic}[1]
    \Require{%
      \pbox[t]{40em}{
        $x_{\mathrm{init}} \in \R^d$;
        $\Linit, M_0 > 0$;
        $\alpha > 1$;
        $0 < \beta \leq 1$\\
        Recommended: $(\Linit, M_0, \alpha, \beta) = (10^{-3}, 10^{-16}, 2, 0.9)$
      }
    }
    \State{%
      $(x_0, y_0) \gets (x_{\mathrm{init}}, x_{\mathrm{init}})$,
      $\Lest \gets \Linit$,
      $k \gets 0$,
      $K \gets 0$
    }
    \Repeat
      \State{%
        $k \gets k + 1$,
        $K \gets K + 1$
      }
      \State $x_k \gets y_{k-1} - \frac{1}{\Lest} \nabla f(y_{k-1})$
      \label{alg-line-agd:update_x}
      \State $y_k \gets x_k + \frac{k}{k+1} (x_k - x_{k-1})$
      \label{alg-line-agd:update_y}

      \If{$f(x_k) > f(x_0) - \frac{\Lest S_k}{2 (k+1)}$}
      \Comment{$S_k \coloneqq \sum_{i=1}^k \norm*{x_i - x_{i-1}}^2$}
      \label{alg-line-agd:check_decrease}
        \State{%
          $(x_0, y_0) \gets (x_{k-1}, x_{k-1})$,
          $\Lest \gets \alpha \Lest$,
          $k \gets 0$
        }
        \label{alg-line-agd:restart_unsuccessful}
      \ElsIf{$(k+1)^5 \Mest_k^2 S_k > \Lest^2$}
      \Comment{$\Mest_k$ is defined by \cref{eq:update_Mest_practical} or \cref{eq:update_Mest}}
      \label{alg-line-agd:check_large_Sk}
        \State{%
          $(x_0, y_0) \gets (x_k, x_k)$,
          $\Lest \gets \beta \Lest$,
          $k \gets 0$
        }
        \label{alg-line-agd:restart_successful}
      \EndIf
    \Until{convergence}
    \State \Return averaged solution $\bar y_k$ defined by \eqref{eq:def_ybark}
  \end{algorithmic}
\end{algorithm}

\paragraph{Restart mechanism 1 (Lines \ref{alg-line-agd:check_decrease} and \ref{alg-line-agd:restart_unsuccessful})}
After updating the solutions, \cref{alg:proposed_agd} checks whether the current estimate $L$ approximates the actual value $L_f$ well.
If
\begin{align}
  f(x_k) - f(x_0)
  \leq
  - \frac{1 - \theta_k}{2} \Lest S_k
  =
  - \frac{\Lest S_k}{2 (k+1)},
  \quad\text{where}\quad
  S_k
  \coloneqq
  \sum_{i=1}^k \norm*{x_i - x_{i-1}}^2,
  \label{eq:descent_condition}
\end{align}
then $L$ is judged suitable, and the epoch continues; otherwise, AGD restarts from $x_{k-1}$ with a larger $L$.
Since condition~\cref{eq:descent_condition} holds in the previous iteration, starting the next epoch from $x_{k-1}$ ensures that this epoch does not increase the objective function value.
Condition~\cref{eq:descent_condition} guarantees a sufficient decrease in the objective function in the epoch and plays a crucial role in our analysis.
As will be shown later, this condition is always satisfied when $L$ is sufficiently large, which justifies the restart mechanism.

\begin{remark}
  This restart mechanism can be regarded as a kind of backtracking.
  A well-known backtracking method for estimating $L_f$ is based on the Armijo rule \cite{armijo1966minimization}, and one might think of using the Armijo rule or its variant \cite[Section~5.3]{becker2011templates} instead of condition~\cref{eq:descent_condition}.
  However, with the Armijo rule, the estimate $L$ changes at each iteration, complicating the complexity analysis of AGD, especially for the nonconvex case.
  In contrast, we fix $L$ through the epoch with the help of condition~\cref{eq:descent_condition}, preserving the simplicity of the existing analysis when $L_f$ is known.
\end{remark}

\paragraph{Restart mechanism 2 (Lines \ref{alg-line-agd:check_large_Sk} and \ref{alg-line-agd:restart_successful})}
Our AGD also resets the effect of acceleration by restart when $S_k$ becomes large; the restart condition involves an estimate $M_k$ of the Hessian's Lipschitz constant $M_f$ (detailed later).
If
\begin{align}
  \Mest_k^2 S_k
  \leq
  \Lest^2 (1 - \theta_k)^5
  =
  \frac{\Lest^2}{(k+1)^5},
  \label{eq:condition_continue}
\end{align}
the epoch continues, otherwise, AGD restarts with the computed solution $x_k$ and a (possibly) smaller $L$.
Condition~\cref{eq:condition_continue} is inspired by a condition in \cite{li2022restarted}, $k M_f S_k \leq \epsilon$; unlike the existing one, our condition does not involve $M_f$ or $\epsilon$.
Although $L$ does not have to be decreased on Line~\ref{alg-line-agd:restart_successful} to derive the complexity bound of $O(\epsilon^{-7/4})$, decreasing $L$ will help the algorithm to capture the local curvature of $f$ and to converge faster empirically.

\paragraph{Averaged solution}
\label{sec:averaged_solution}
Our AGD guarantees that the gradient norm is small at a weighted average $\bar y_k$ of the solutions $y_0,\dots,y_{k-1}$.
The averaged solution $\bar y_k$ is defined by
\begin{align}
  \bar y_k
  \coloneqq
  \sum_{i=0}^{k-1} p_{k,i} y_i,
  \quad\text{where}\quad
  p_{k,i}
  \coloneqq
  \frac{1}{Z_k} \prod_{j=i+1}^{k-1} \theta_j,
  \quad
  Z_k
  \coloneqq
  \sum_{i=0}^{k-1} \prod_{j=i+1}^{k-1} \theta_j.
  \label{eq:def_ybark}
\end{align}
Note that it is unnecessary to keep all $y_0,\dots,y_{k-1}$ in memory to compute $\bar y_k$ because $Z_k$ and $\bar y_k$ satisfy a simple recursion:
\begin{align}
  Z_{k+1}
  &=
  1 + \theta_k Z_k,\quad
  \bar y_{k+1}
  =
  \frac{1}{Z_{k+1}} \prn*{y_k + \theta_k Z_k \bar y_k},
\end{align}
which maintains the computational efficiency of the algorithm.
Given \cref{eq:def_theta_k}, there is another form:
\begin{align}
  \bar y_k
  =
  \frac{2}{k (k + 1)}
  \sum_{i=0}^{k-1} (i + 1) y_i,\quad
  Z_k
  =
  \frac{k + 1}{2},
  \label{eq:ybar_Zk_alternative}
\end{align}
and this simplification of $\bar y_k$ and $Z_k$ is one of the aims of setting $\theta_k$ as in \cref{eq:def_theta_k}.

\paragraph{Estimate of Hessian's Lipschitz constant}
The estimate $M_k$ for the Hessian is updated at each iteration, unlike $L$ for the gradient.
Our complexity analysis requires three inequalities on $M_k$: $M_{k-1} \leq M_k$ and
\begin{align}
  &
  f(y_k) - f(x_k)
  \leq
  \frac{1}{2} \inner*{\nabla f(y_k) + \nabla f(x_k)}{y_k - x_k}
  + \frac{\Mest_k}{12} \norm*{y_k - x_k}^3,
  \label{eq:fy_fx_diff_upperbound_Mestk}\\
  &
  \norm*{
    \nabla f(y_k)
    + \theta_k \nabla f(x_{k-1})
    - (1 + \theta_k) \nabla f (x_k)
  }
  \leq
  \theta_k \Mest_k \norm*{x_k - x_{k-1}}^2.
  \label{eq:grad_interpolation_error_Mestk}
\end{align}
Finding the smallest $M_k$ that satisfies these inequalities is desirable, and fortunately, it is straightforward,
\begin{align}
  \label{eq:update_Mest_practical}
  \Mest_k
  =
  \max \bigg\{
    \Mest_{k-1},\ 
    &
    \frac{
      12
      \prn*{
        f(y_k) - f(x_k) 
        - \frac{1}{2} \inner*{\nabla f(y_k) + \nabla f(x_k)}{y_k - x_k}
      }
    }{\norm*{y_k - x_k}^3},\\
    &
    \frac{
      \norm*{
        \nabla f(y_k)
        + \theta_k \nabla f(x_{k-1})
        - (1 + \theta_k) \nabla f (x_k)
      }
    }{\theta_k \norm*{x_k - x_{k-1}}^2}
  \bigg\}
\end{align}
because $x_k$, $y_k$, and $\theta_k$ are already in hand.
In other words, our $\theta_k$ in \cref{eq:def_theta_k} does not depend on the estimate $M_k$ of $M_f$, which facilitates computation of $M_k$.
Thus, we can estimate $M_f$ without backtracking, thereby simplifying the algorithm and improving its efficiency.

Instead of \cref{eq:update_Mest_practical}, we can use
\begin{align}
  \label{eq:update_Mest}
  \Mest_k
  =
  \max \bigg\{
    \Mest_{k-1},\ 
    &
    \frac{
      12
      \prn*{
        f(y_k) - f(x_k) 
        - \frac{1}{2} \inner*{\nabla f(y_k) + \nabla f(x_k)}{y_k - x_k}
      }
    }{\norm*{y_k - x_k}^3}
    ,\\
    &
    \frac{
      \norm*{
        \nabla f(y_k)
        + \theta_k \nabla f(x_{k-1})
        - (1 + \theta_k) \nabla f (x_k)
      }
    }{\theta_k \norm*{x_k - x_{k-1}}^2},\\
    &
    \frac{
      16
      \prn*{
        Z_k^2 \norm*{ \nabla f (\bar y_k) }
        - Z_k \Lest \norm*{x_k - x_{k-1}}
      }
    }{(k-1) (k+5)^2 S_k}
  \bigg\}.
\end{align}
in the algorithm.\footnote{
  For $k = 1$, \Cref{eq:update_Mest} takes the form $M_1 = \max \set{M_0, \ast, \ast, \frac{0}{0}}$, which should be treated as $M_1 = \max \set{M_0, \ast, \ast}$.
}
This replacement ensures that
\begin{align}
  \norm*{ \nabla f (\bar y_k) }
  \leq
  \frac{\Lest}{Z_k} \norm*{x_k - x_{k-1}}
  + \frac{(k-1) (k+5)^2}{16 Z_k^2} \Mest_k S_k
  \label{eq:grad_ybark_upperbound_Mestk}
\end{align}
and will lead to a smaller complexity bound with lessened dependence on $M_0$ and $M_f$, as will be confirmed later in \cref{thm:complexity}.
\Cref{eq:update_Mest} derives a theoretically better bound than \cref{eq:update_Mest_practical} but requires an additional gradient evaluation at $\bar y_k$, increasing the computational cost per iteration.
Therefore, we recommend using \cref{eq:update_Mest_practical} in practice.

\section{Complexity analysis}
\label{sec:analysis}
This section provides the complexity bounds of $O(\epsilon^{-7/4})$ for \cref{alg:proposed_agd}.

\subsection{Upper bound on $L$}
The estimates $L$ and $M_k$ of the Lipschitz constants should be large enough to satisfy some technical inequalities such as \cref{eq:descent_condition,eq:fy_fx_diff_upperbound_Mestk,eq:grad_interpolation_error_Mestk}, but they should not be too large.
This section proves the following upper bound on $L$.
\begin{proposition}
  \label{prop:L_upperbound}
  Suppose that \cref{asm:gradient_lip} holds.
  Then, the following is true throughout \cref{alg:proposed_agd}: $\Lest \leq \max\set{\Linit, \alpha \Ltrue}$.
\end{proposition}
This proposition immediately follows from the following lemma.
\begin{lemma}
  \label{lem:decrease_epoch}
  Suppose that \cref{asm:gradient_lip} holds.
  During epochs with $\Lest \geq \Ltrue$, condition~\cref{eq:descent_condition} always holds.
\end{lemma}
Before providing the proof, let us describe the idea behind it.

To obtain the descent condition~\cref{eq:descent_condition}, we first evaluate the decrease for one iteration.
However, a common difficulty arises when analyzing AGD: the objective function value does not necessarily decrease monotonically.
To deal with the problem, we introduce a potential function related to $f(x_k)$ and show that it is \emph{nearly} decreasing.
The potential function $\Phi_k$ is defined by
\begin{align}
  \Phi_k
  \coloneqq
  f(x_k)
  + \frac{\theta_k^2}{2} \prn*{
    \inner*{\nabla f(x_{k-1})}{x_k - x_{k-1}}
    + \frac{1}{2 \Lest} \norm*{\nabla f(x_{k-1})}^2
    + \Lest \norm*{x_k - x_{k-1}}^2
  },
  \label{eq:def_potential}
\end{align}
inspired by \cref{eq:hint_potential} used in the proof.
The value $\Phi_k$ decreases when $\norm*{x_k - x_{k-1}}$ is small enough as the following lemma shows.
\begin{lemma}
  \label{lem:potential_decrease_iteration}
  Under \cref{asm:gradient_lip} and $\Lest \geq \Ltrue$, the following holds for all $k \geq 0$:
  \begin{align}
    \Phi_{k+1} - \Phi_k
    &\leq
    \frac{\theta_{k+1}^2 + \theta_k - 2}{4} \Lest \norm*{x_{k+1} - x_k}^2
    + \frac{7 \theta_k^2}{12} \Mest_k \norm*{x_k - x_{k-1}}^3\\
    &\quad
    + \frac{\theta_k^3}{4 \Lest} \Mest_k^2 \norm*{x_k - x_{k-1}}^4
    - \frac{\theta_k^2}{4 \Lest} \norm*{\nabla f (x_k)}^2.
    \label{eq:potential_decrease_iteration}
  \end{align}
\end{lemma}
This lemma can be proven by putting \cref{lem:eq:property_gradlip_obj} due to the Lipschitz gradient together with inequalities \cref{eq:fy_fx_diff_upperbound_Mestk,eq:grad_interpolation_error_Mestk} due to the Lipschitz Hessian.
Summing \cref{lem:potential_decrease_iteration} over $k$ and doing some calculations yields \cref{lem:decrease_epoch}.
To evaluate the terms of $\norm*{x_k - x_{k-1}}^3$ and $\norm*{x_k - x_{k-1}}^4$ in \cref{lem:potential_decrease_iteration} in these calculations, we use 
\begin{align}
  \Mest_{k-1}^2 S_{k-1}
  \leq
  \Lest^2 (1 - \theta_{k-1})^5
  =
  \frac{\Lest^2}{k^5},
  \label{eq:condition_continue_k-1}
\end{align}
which follows from the fact that the epoch did not end at iteration $k-1$.

Now, we prove \cref{lem:potential_decrease_iteration} and use it to prove \cref{lem:decrease_epoch}.
The following proofs do not use the specific form \cref{eq:def_theta_k} of $\theta_k$ but rather a more general condition on $\theta_k$,
\begin{align}
  \theta_{k+1}^2 \leq \theta_k \leq \theta_{k+1},
  \label{eq:theta_increase_not_too_fast}
\end{align}
which is easily verified for all $k \geq 0$ under $\theta_0 \coloneqq \theta_1^2$ and \cref{eq:def_theta_k} for $k \geq 1$.

\begin{proof}[Proof of \cref{lem:potential_decrease_iteration}]
  In the proof, let
  \begin{align}
    P_k
    &\coloneqq
    \inner*{\nabla f(x_{k-1})}{x_k - x_{k-1}}
  \end{align}
  to simplify the notation.
  From \cref{lem:eq:property_gradlip_obj,eq:update_xy}, we have
  \begin{align}
    f(x_{k+1}) - f(y_k)
    &\leq
    \inner*{\nabla f(y_k)}{x_{k+1} - y_k}
    + \frac{\Lest}{2} \norm*{x_{k+1} - y_k}^2
    = 
    - \frac{1}{2 \Lest} \norm*{\nabla f(y_k)}^2.
  \end{align}
  Summing this bound and \cref{eq:fy_fx_diff_upperbound_Mestk} yields
  \begin{alignat}{2}
    &\mathInd
    f(x_{k+1}) - f(x_k)
    \label{eq:descent_f_first}\\
    &\leq
    \frac{1}{2} \inner*{\nabla f(y_k) + \nabla f(x_k)}{y_k - x_k}
    + \frac{\Mest_k}{12} \norm*{y_k - x_k}^3
    - \frac{1}{2 \Lest} \norm*{\nabla f(y_k)}^2\\
    &\leq
    \frac{1}{2} \inner*{\nabla f(y_k) + \nabla f(x_k)}{y_k - x_k}
    + \frac{\theta_k^2 \Mest_k}{12} \norm*{x_k - x_{k-1}}^3
    - \frac{1}{2 \Lest} \norm*{\nabla f(y_k)}^2,
  \end{alignat}
  where the last inequality uses \cref{eq:update_xy} and $0 \leq \theta_k \leq 1$.
  To evaluate the first term on the right-hand side, we decompose it into four terms:
  \begin{align}
    &\mathInd
    \inner*{\nabla f(y_k) + \nabla f(x_k)}{y_k - x_k}\\
    &=
    \underbrace{2 \inner*{\nabla f(y_k)}{y_k - x_k}}_{\text{(A)}}
    {}+{} \underbrace{\theta_k \inner*{\nabla f(x_{k-1})}{y_k - x_k}}_{\text{(B)}}\\[-0.8\baselineskip]
    &\quad
    - \underbrace{\theta_k \inner*{\nabla f(x_k)}{y_k - x_k}}_{\text{(C)}}
    {}-{} \underbrace{
        \inner*{
        \nabla f(y_k)
        + \theta_k \nabla f(x_{k-1})
        - (1 + \theta_k) \nabla f (x_k)
      }{y_k - x_k}
    }_{\text{(D)}}.
  \end{align}
  Each term can be evaluated as
  \begin{alignat}{2}
    \text{(A)}
    &=
    \frac{1}{\Lest} \norm*{\nabla f(y_k)}^2
    + \Lest \norm*{y_k - x_k}^2
    - \Lest \norm*{(y_k - x_k) - \frac{1}{\Lest} \nabla f(y_k)}^2\\
    &=
    \frac{1}{\Lest} \norm*{\nabla f(y_k)}^2
    + \theta_k^2 \Lest \norm*{x_k - x_{k-1}}^2
    - \Lest \norm*{x_{k+1} - x_k}^2
    &\quad&\by{\cref{eq:update_xy}},\\
    \text{(B)}
    &=
    \theta_k^2 \inner*{\nabla f(x_{k-1})}{x_k - x_{k-1}}
    =
    \theta_k^2 P_k
    &\quad&\by{\cref{eq:update_xy}},\\
    \text{(C)}
    &=
    - \theta_k P_{k+1}
    + \theta_k \inner*{\nabla f(x_k)}{x_{k+1} - y_k}\\
    &=
    - \theta_k P_{k+1}
    - \frac{\theta_k}{\Lest} \inner*{\nabla f(x_k)}{\nabla f(y_k)}
    &\quad&\by{\cref{eq:update_xy}},\\
    \text{(D)}
    &\leq
    \theta_k \Mest_k \norm*{x_k - x_{k-1}}^2 \norm*{y_k - x_k}
    &\quad&\by{\cref{eq:grad_interpolation_error_Mestk}}\\
    &=
    \theta_k^2 \Mest_k \norm*{x_k - x_{k-1}}^3
    &\quad&\by{\cref{eq:update_xy}}.
  \end{alignat}
  Plugging the evaluations into \cref{eq:descent_f_first} results in
  \begin{align}
    f(x_{k+1}) - f(x_k)
    \label{eq:descent_f_second}
    &\leq
    \frac{\theta_k^2 \Lest}{2} \norm*{x_k - x_{k-1}}^2
    - \frac{\Lest}{2} \norm*{x_{k+1} - x_k}^2
    + \frac{\theta_k^2}{2} P_k
    - \frac{\theta_k}{2} P_{k+1}\\
    &\quad
    + \frac{7 \theta_k^2}{12} \Mest_k \norm*{x_k - x_{k-1}}^3
    - \frac{\theta_k}{2 \Lest} \inner*{\nabla f(x_k)}{\nabla f(y_k)}.
  \end{align}

  Next, to bound the last term on the right-hand side, we use the following inequality from \cref{eq:grad_interpolation_error_Mestk}:
  \begin{align}
    \norm*{
      \nabla f(y_k)
      - (1 + \theta_k) \nabla f (x_k)
    }
    &\leq
    \theta_k \norm*{ \nabla f(x_{k-1})}
    + 
    \theta_k \Mest_k
    \norm*{x_k - x_{k-1}}^2.
  \end{align}
  We square both sides as
  \begin{alignat}{2}
    &\mathInd
    \norm*{
      \nabla f(y_k)
      - (1 + \theta_k) \nabla f (x_k)
    }^2\\
    &=
    \norm*{\nabla f(y_k)}^2
    + (1 + \theta_k)^2 \norm*{\nabla f (x_k)}^2
    - 2 (1 + \theta_k) \inner*{\nabla f (x_k)}{\nabla f(y_k)}\\
    &\geq
    (1 + \theta_k)^2 \norm*{\nabla f (x_k)}^2
    - 2 (1 + \theta_k) \inner*{\nabla f (x_k)}{\nabla f(y_k)}
  \end{alignat}
  and
  \begin{alignat}{2}
    \theta_k^2 \prn*{
      \norm*{ \nabla f(x_{k-1})}
      + \Mest_k \norm*{x_k - x_{k-1}}^2
    }^2
    &\leq
    \theta_k (1 + \theta_k) \norm*{ \nabla f(x_{k-1})}^2\\
    &\quad
    + \theta_k^2 (1 + \theta_k) \Mest_k^2 \norm*{x_k - x_{k-1}}^4,
  \end{alignat}
  where we have used $(a + b)^2 \leq (1 + 1/\theta) a^2 + (1 + \theta) b^2$ for $a, b, \theta > 0$.
  Rearranging the terms yields
  \[
    - \inner*{\nabla f (x_k)}{\nabla f(y_k)}
    \leq
    \frac{\theta_k}{2} \norm*{ \nabla f(x_{k-1})}^2
    - \frac{1 + \theta_k}{2} \norm*{\nabla f (x_k)}^2
    + \frac{\theta_k^2}{2} \Mest_k^2 \norm*{x_k - x_{k-1}}^4.
  \]
  Plugging this bound into \cref{eq:descent_f_second} results in
  \begin{align}
    \label{eq:hint_potential}
    &\mathInd
    f(x_{k+1}) - f(x_k)\\
    &\leq
    \frac{\theta_k^2}{2} \prn*{
      P_k
      + \frac{1}{2 \Lest} \norm*{\nabla f(x_{k-1})}^2
      + \Lest \norm*{x_k - x_{k-1}}^2
    }\\
    &\quad
    -
    \frac{\theta_k}{2} \prn*{
      P_{k+1} + \frac{1}{2 \Lest} \norm*{\nabla f(x_k)}^2
    }
    - \frac{\Lest}{2} \norm*{x_{k+1} - x_k}^2\\
    &\quad
    + \frac{7 \theta_k^2}{12} \Mest_k \norm*{x_k - x_{k-1}}^3
    + \frac{\theta_k^3}{4 \Lest} \Mest_k^2 \norm*{x_k - x_{k-1}}^4
    - \frac{\theta_k^2}{4 \Lest} \norm*{\nabla f (x_k)}^2.
  \end{align}
  This inequality can be rewritten with the potential function $\Phi_k$ defined in \cref{eq:def_potential} as
  \begin{align}
    \Phi_{k+1} - \Phi_k
    &\leq
    \frac{\theta_{k+1}^2 - \theta_k}{2}
    \prn*{
      P_{k+1}
      + \frac{1}{2 \Lest} \norm*{\nabla f(x_k)}^2
    }
    + \frac{\theta_{k+1}^2 - 1}{2} \Lest \norm*{x_{k+1} - x_k}^2\\
    &\quad
    + \frac{7 \theta_k^2}{12} \Mest_k \norm*{x_k - x_{k-1}}^3
    + \frac{\theta_k^3}{4 \Lest} \Mest_k^2 \norm*{x_k - x_{k-1}}^4
    - \frac{\theta_k^2}{4 \Lest} \norm*{\nabla f (x_k)}^2.
  \end{align}
  Finally, using $\theta_{k+1}^2 - \theta_k \leq 0$ from \cref{eq:theta_increase_not_too_fast} and
  \begin{align}
    - P_{k+1}
    =
    - \inner*{\nabla f(x_k)}{x_{k+1} - x_k}
    \leq
    \frac{1}{2 \Lest} \norm*{\nabla f(x_k)}^2
    + \frac{\Lest}{2} \norm*{x_{k+1} - x_k}^2
  \end{align}
  from Young's inequality ($\abs{\inner*{a}{b}} \leq \frac{1}{2} \norm*{a}^2 + \frac{1}{2} \norm*{b}^2$) yields the desired result as
  \begin{align}
    \Phi_{k+1} - \Phi_k
    &\leq
    \frac{\theta_{k+1}^2 - \theta_k}{2} \prn*{- \frac{\Lest}{2} \norm*{x_{k+1} - x_k}^2}
    + \frac{\theta_{k+1}^2 - 1}{2} \Lest \norm*{x_{k+1} - x_k}^2\\
    &\quad
    + \frac{7 \theta_k^2}{12} \Mest_k \norm*{x_k - x_{k-1}}^3
    + \frac{\theta_k^3}{4 \Lest} \Mest_k^2 \norm*{x_k - x_{k-1}}^4
    - \frac{\theta_k^2}{4 \Lest} \norm*{\nabla f (x_k)}^2\\
    &=
    \frac{\theta_{k+1}^2 + \theta_k - 2}{4} \Lest \norm*{x_{k+1} - x_k}^2
    + \frac{7 \theta_k^2}{12} \Mest_k \norm*{x_k - x_{k-1}}^3\\
    &\quad
    + \frac{\theta_k^3}{4 \Lest} \Mest_k^2 \norm*{x_k - x_{k-1}}^4
    - \frac{\theta_k^2}{4 \Lest} \norm*{\nabla f (x_k)}^2.
  \end{align}
\end{proof}

\begin{proof}[Proof of \cref{lem:decrease_epoch}]
  Note that
  \begin{align}
    \sum_{i=1}^k \norm*{x_i - x_{i-1}}^3
    &\leq
    S_k^{3/2},\quad
    \sum_{i=1}^k \norm*{x_i - x_{i-1}}^4
    \leq
    S_k^2
    \label{eq:sum_xk_diff_power34_upperbound}
  \end{align}
  by the monotonicity of $\ell_p$-norms, i.e., $\norm{a}_p \geq \norm{a}_q$ for $1 \leq p \leq q$.
  In addition, we have
  \begin{align}
    \Phi_k - f(x_k)
    &=
    \frac{\theta_k^2}{2} \prn*{
      \frac{1}{2 \Lest} \norm*{\nabla f(x_{k-1}) + L (x_k - x_{k-1})}^2
      + \frac{\Lest}{2} \norm*{x_k - x_{k-1}}^2
    }
    \geq
    0,
    \label{eq:potential_lowerbound}\\
    \Phi_0 - f(x_0)
    &=
    \frac{\theta_0^2}{4 \Lest} \norm*{\nabla f(x_0)}^2
    \label{eq:potential_obj_0}
  \end{align}
  since $x_{-1} = x_0$.

  Summing \cref{lem:potential_decrease_iteration} over $k$ and telescoping yields
  \begin{align}
    &\mathInd
    \Phi_k - \Phi_0
    \label{eq:potential_decrease_epoch}\\
    &\leq
    \sum_{i=0}^{k-1}
    \bigg(
      \frac{\theta_{i+1}^2 + \theta_i - 2}{4} \Lest \norm*{x_{i+1} - x_i}^2
      + \frac{7 \theta_i^2}{12} \Mest_i \norm*{x_i - x_{i-1}}^3\\
      &\qquad\qquad
      + \frac{\theta_i^3}{4 \Lest} \Mest_i^2 \norm*{x_i - x_{i-1}}^4
      - \frac{\theta_i^2}{4 \Lest} \norm*{\nabla f(x_i)}^2
    \bigg)\\
    &\leq
    \frac{\theta_k^2 + \theta_{k-1} - 2}{4} \Lest S_k
    + \frac{7 \theta_{k-1}^2}{12} \Mest_{k-1} S_{k-1}^{3/2} 
    + \frac{\theta_{k-1}^3}{4 \Lest} \Mest_{k-1}^2 S_{k-1}^2
    - \frac{\theta_0^2}{4 \Lest} \norm*{\nabla f(x_0)}^2,
  \end{align}
  where we have used \cref{eq:theta_increase_not_too_fast}, \cref{eq:sum_xk_diff_power34_upperbound}, and $\Mest_i \leq \Mest_{k-1}$.
  Combining \cref{eq:potential_lowerbound,eq:potential_obj_0,eq:potential_decrease_epoch} yields
  \begin{align}
    f(x_k) - f(x_0)
    \leq
    \frac{\theta_k^2 + \theta_{k-1} - 2}{4} \Lest S_k
    + \frac{7 \theta_{k-1}^2}{12} \Mest_{k-1} S_{k-1}^{3/2} 
    + \frac{\theta_{k-1}^3}{4 \Lest} \Mest_{k-1}^2 S_{k-1}^2.
  \end{align}
  Furthermore, we bound the last two terms as
  \begin{alignat}{2}
    &\mathInd
    \frac{7 \theta_{k-1}^2}{12} \Mest_{k-1} S_{k-1}^{3/2} 
    + \frac{\theta_{k-1}^3}{4 \Lest} \Mest_{k-1}^2 S_{k-1}^2\\
    &\leq
    \prn*{
      \frac{7 \theta_{k-1}^2}{12} \Mest_{k-1} \sqrt{S_{k-1}}
      + \frac{\theta_{k-1}^3}{4 \Lest} \Mest_{k-1}^2 S_{k-1}
    }
    S_k
    &\quad&\by{$S_{k-1} \leq S_k$}\\
    &\leq
    \prn*{
      \frac{7}{12} \theta_{k-1}^2 (1 - \theta_{k-1})^{5/2}
      + \frac{1}{4} \theta_{k-1}^3 (1 - \theta_{k-1})^5
    }
    \Lest S_k
    &\quad&\by{\cref{eq:condition_continue_k-1}}\\
    &\leq
    \prn*{
      \frac{7}{12} + \frac{1}{4}
    }
    \theta_{k-1}^2 (1 - \theta_{k-1})^2
    \Lest S_k
    &\quad&\by{$0 \leq \theta_{k-1} \leq 1$}\\
    &\leq
    \frac{1}{4}
    \theta_{k-1} (1 - \theta_{k-1})
    \Lest S_k
    &\quad&\by{$0 \leq \theta_{k-1} \leq 1$},
  \end{alignat}
  which results in
  \begin{alignat}{2}
    f(x_k) - f(x_0)
    &\leq
    \frac{\theta_k^2 + \theta_{k-1} - 2}{4} \Lest S_k
    + \frac{1}{4} \theta_{k-1} (1 - \theta_{k-1}) \Lest S_k\\
    &=
    - \frac{1 - \theta_k}{2} \Lest S_k
    + \frac{(1 - \theta_k)^2 - (1 - \theta_{k-1})^2}{4} \Lest S_k\\
    &\leq
    - \frac{1 - \theta_k}{2} \Lest S_k
    &\quad&\by{\cref{eq:theta_increase_not_too_fast}}
  \end{alignat}
  and concludes the proof.
\end{proof}

\subsection{Upper bound on $M_k$}
This section proves the following upper bound on $M_k$.
\begin{proposition}
  \label{prop:Mk_upperbound}
  Suppose that \cref{asm:hessian_lip} holds.
  Then, the following is true throughout \cref{alg:proposed_agd}: 
  $\Mest_k \leq \max\set{M_0, \Mtrue}$.
\end{proposition}
Because of the update rules \cref{eq:update_Mest_practical} or \cref{eq:update_Mest} of $M_k$, it suffices to show the inequalities obtained by replacing $\Mest_k$ with $M_f$ in \cref{eq:fy_fx_diff_upperbound_Mestk,eq:grad_interpolation_error_Mestk,eq:grad_ybark_upperbound_Mestk}.
One of them has already been obtained as \cref{lem:trapezoidal_rule_error}.
Below we prove the other two; the following lemma gives the formal statement.
\begin{lemma}
  \label{lem:grad_interpolation_error}
  Under \cref{asm:hessian_lip}, the following hold for all $k \geq 1$:
  \begin{align}
    &
    \norm*{
      \nabla f(y_k)
      + \theta_k \nabla f(x_{k-1})
      - (1 + \theta_k) \nabla f (x_k)
    }
    \leq
    \theta_k \Mtrue \norm*{x_k - x_{k-1}}^2,
    \label{eq:grad_interpolation_error_Mf}\\
    &
    \norm*{ \nabla f (\bar y_k) }
    \leq
    \frac{\Lest}{Z_k} \norm*{x_k - x_{k-1}}
    + \frac{(k-1) (k+5)^2}{16 Z_k^2} \Mtrue S_k.
    \label{eq:grad_ybark_upperbound_Mf}
  \end{align}
\end{lemma}
The following proofs do not use the specific form \cref{eq:def_theta_k} of $\theta_k$; they only use $0 \leq \theta_k \leq 1$.

\begin{proof}[Proof of \cref{eq:grad_interpolation_error_Mf}]
  Note that $x_k = \frac{1}{1 + \theta_k} y_k + \frac{\theta_k}{1 + \theta_k} x_{k-1}$ from \cref{eq:update_xy}, and \cref{lem:gradient_jensen} gives
  \begin{alignat}{2}
    &\mathInd
    \norm*{
      \nabla f (x_k)
      - \frac{1}{1 + \theta_k} \nabla f(y_k)
      - \frac{\theta_k}{1 + \theta_k} \nabla f(x_{k-1})
    }\\
    &\leq
    \frac{\Mtrue \theta_k}{2 (1 + \theta_k)^2} \norm*{y_k - x_{k-1}}^2\\
    &=
    \frac{\Mtrue \theta_k}{2} \norm*{x_k - x_{k-1}}^2
    \leq
    \frac{\Mtrue \theta_k}{1 + \theta_k} \norm*{x_k - x_{k-1}}^2
    &\quad&\by{\cref{eq:update_xy} and $0 \leq \theta_k \leq 1$}.
  \end{alignat}
  Multiplying by $(1 + \theta_k)$ concludes the proof.
\end{proof}
\begin{proof}[Proof of \cref{eq:grad_ybark_upperbound_Mf}]
  We have
  \begin{alignat}{2}
    &\mathInd
    p_{k,i} \nabla f(y_i)
    =
    p_{k,i} \Lest (y_i - x_{i+1})\\
    &=
    p_{k,i} \Lest \prn*{
      \theta_i (x_i - x_{i-1})
      - (x_{i+1} - x_i)
    }
    &\quad&\by{\cref{eq:update_xy}}\\
    &=
    \Lest \prn*{
      p_{k,i-1} (x_i - x_{i-1})
      - p_{k,i} (x_{i+1} - x_i)
    }
    &\quad&\by{definition \cref{eq:def_ybark} of $p_{k,i}$}.
  \end{alignat}
  Summing this equality over $i$ yields
  \begin{align}
    \sum_{i=0}^{k-1}
    p_{k,i} \nabla f(y_i)
    =
    - \Lest p_{k,k-1} (x_k - x_{k-1}),
    \label{eq:weighted_average_gradient_equation}
  \end{align}
  where we have used $x_0 = x_{-1}$.
  Now, we obtain
  \begin{alignat}{2}
    &\mathInd
    \norm*{
      \nabla f (\bar y_k)
    }
    \label{eq:norm_grad_ybark_upperbound_pre}\\
    &\leq
    \norm*{
      \sum_{i=0}^{k-1} p_{k,i} \nabla f \prn*{ y_i }
    }
    + \frac{\Mtrue}{2} \sum_{0 \leq i < j < k} p_{k,i} p_{k,j} \norm*{y_i - y_j}^2
    &\quad&\by{\cref{lem:gradient_jensen}}\\
    &=
    \Lest p_{k,k-1} \norm*{x_k - x_{k-1}}
    + \frac{\Mtrue}{2} \sum_{0 \leq i < j < k} p_{k,i} p_{k,j} \norm*{y_i - y_j}^2
    &\quad&\by{\cref{eq:weighted_average_gradient_equation}}\\
    &\leq
    \frac{\Lest}{Z_k} \norm*{x_k - x_{k-1}}
    + \frac{\Mtrue}{2 Z_k^2} \sum_{0 \leq i < j < k} \norm*{y_i - y_j}^2,
  \end{alignat}
  where the last inequality follows from $p_{k,i} \leq p_{k,k-1} = 1 / Z_k$ for all $0 \leq i < k$.

  Next, we bound the second term on the right-hand side.
  Since for $0 \leq i < j < k$,
  \begin{alignat}{2}
    &\mathInd
    \norm*{y_i - y_j}\\
    &\leq
    \norm*{y_i - x_i}
    + \sum_{l=i+1}^{j-1} \norm*{x_l - x_{l-1}}
    + \norm*{y_j - x_{j-1}}
    &\quad&\by{the triangle inequality}\\
    &=
    \norm*{x_i - x_{i-1}}
    + \sum_{l=i+1}^{j-1} \norm*{x_l - x_{l-1}}
    + 2 \norm*{x_j - x_{j-1}}
    &\quad&\by{\cref{eq:update_xy} and $0 \leq \theta_k \leq 1$}\\
    &\leq
    \prn*{
      1^2 + \sum_{l=i+1}^{j-1} 1^2 + 2^2
    }^{1/2}
    \prn*{
      \sum_{l=i}^j \norm*{x_l - x_{l-1}}^2
    }^{1/2}
    &\quad&\by{Cauchy--Schwarz}\\
    &=
    \sqrt{j - i + 4}
    \prn*{
      \sum_{l=i}^j \norm*{x_l - x_{l-1}}^2
    }^{1/2},
  \end{alignat}
  we have
  \begin{align}
    \sum_{0 \leq i < j < k}
    \norm*{y_i - y_j}^2
    &\leq
    \sum_{0 \leq i < j < k}
    \sum_{l=i}^j
    (j - i + 4) \norm*{x_l - x_{l-1}}^2\\
    &=
    \sum_{l=0}^{k-1}
    \prn*{
      \sum_{i=0}^l
      \sum_{j=l}^{k-1}
      (j - i + 4) 
    }
    \norm*{x_l - x_{l-1}}^2
    - 4 \sum_{l=0}^{k-1} \norm*{x_l - x_{l-1}}^2\\
    &=
    \frac{k + 7}{2}
    \sum_{l=0}^{k-1}
    (l + 1) (k - l)
    \norm*{x_l - x_{l-1}}^2
    - 4 \sum_{l=0}^{k-1} \norm*{x_l - x_{l-1}}^2\\
    &\leq
    \frac{k + 7}{2}
    \sum_{l=0}^{k-1} \frac{(k+1)^2}{4} \norm*{x_l - x_{l-1}}^2
    - 4 \sum_{l=0}^{k-1} \norm*{x_l - x_{l-1}}^2\\
    &=
    \frac{(k-1) (k+5)^2}{8}
    \sum_{l=0}^{k-1} \norm*{x_l - x_{l-1}}^2
    \leq
    \frac{(k-1) (k+5)^2}{8} S_k.
  \end{align}
  Plugging this bound into \cref{eq:norm_grad_ybark_upperbound_pre} concludes the proof.
\end{proof}

\subsection{Upper bound on gradient norm}
Let $\Lmax$ and $\Mmax$ be the upper bounds on $L$ and $M_k$ given in \cref{prop:L_upperbound,prop:Mk_upperbound}: $\Lmax \coloneqq \max\set{\Linit, \alpha \Ltrue}$ and $\Mmax \coloneqq \max\set{M_0, \Mtrue}$.
The following lemma provides upper bounds on the gradient norm.
\begin{lemma}
  \label{lem:grad_norm_ybar_upperbound}
  Suppose that \cref{asm:gradient_hessian_lip} holds.
  In \cref{alg:proposed_agd}, the following is true when $k \geq 2$:
  \begin{align}
    \min_{1 \leq i < k} \norm*{ \nabla f (\bar y_i) }
    &\leq
    \frac{4 \Lest \Mmax}{\Mest_{k-1}}
    \sqrt{\frac{S_{k-1}}{k^3}}
    \leq
    \frac{4 \Lmax^2 \Mmax}{M_0^2 k^4}.
    \label{eq:grad_norm_ybar_upperbound}
  \end{align}
  If $M_k$ is computed with \cref{eq:update_Mest} in the algorithm, the above bound is improved to
  \begin{align}
    \min_{1 \leq i < k} \norm*{ \nabla f (\bar y_i) }
    &\leq
    4 \Lest \sqrt{\frac{S_{k-1}}{k^3}}
    \leq
    \frac{4 \Lmax^2}{M_0 k^4}.
    \label{eq:grad_norm_ybar_upperbound_tight}
  \end{align}
\end{lemma}
\begin{proof}[Proof of \cref{eq:grad_norm_ybar_upperbound}] 
  Note that
  \begin{align}
    &\sum_{i=1}^{k-1}
    \frac{(i-1)(i+5)^2}{16}
    =
    \frac{(k - 2) (k - 1) (k^2 + 13 k + 50)}{64}
    \leq
    \frac{k^3 (k+1)}{24},
    \label{eq:i_cubic_sum}\\
    &\sum_{i=1}^{k-1} Z_i^2
    =
    \sum_{i=1}^{k-1} \prn*{\frac{i+1}{2}}^2
    =
    \frac{(k-1) (2 k^2 + 5 k + 6)}{24}
    \geq 
    \frac{k^2 (k+1)}{12}
    \label{eq:Zi_squared_sum}
  \end{align}
  from \cref{eq:ybar_Zk_alternative} and $k \geq 2$.
  Using \cref{eq:grad_ybark_upperbound_Mf}, we have
  \begin{align}
    \prn*{
      \sum_{i=1}^{k-1} Z_i^2
    }
    \min_{1 \leq i < k} \norm*{ \nabla f (\bar y_i) }
    &\leq
    \sum_{i=1}^{k-1}
    Z_i^2
    \norm*{ \nabla f (\bar y_i) }\\
    &\leq
    \sum_{i=1}^{k-1}
    \prn*{
      \Lest Z_i \norm*{x_i - x_{i-1}}
      + \frac{(i-1) (i+5)^2}{16} \Mtrue S_i
    }.
  \end{align}
  Each term on the right-hand side can be bounded as follows:
  \begin{alignat}{2}
    &\sum_{i=1}^{k-1} \Lest Z_i \norm*{x_i - x_{i-1}}
    \leq
    \Lest
    \prn*{
      \sum_{i=1}^{k-1} Z_i^2
    }^{1/2}
    \sqrt{S_{k-1}}
    &\quad&\by{Cauchy--Schwarz},\\
    &\sum_{i=1}^{k-1}
    \frac{(i-1)(i+5)^2}{16} \Mtrue S_i
    \leq
    \frac{k^3 (k+1)}{24}
    \Mtrue S_{k-1}
    &\quad&\by{\cref{eq:i_cubic_sum}}.
  \end{alignat}
  We thus have
  \begin{alignat}{2}
    &\mathInd
    \min_{1 \leq i < k} \norm*{ \nabla f (\bar y_i) }\\
    &\leq
    \Lest 
    \prn*{
      \sum_{i=1}^{k-1} Z_i^2
    }^{-1/2}\!\!
    \sqrt{S_{k-1}}
    + 
    \prn*{
      \sum_{i=1}^{k-1} Z_i^2
    }^{-1}
    \frac{k^3 (k+1)}{24}
    \Mtrue S_{k-1}\\
    &\leq
    \Lest 
    \sqrt{\frac{12 S_{k-1}}{k^3}}
    +
    \frac{k}{2}
    \Mtrue S_{k-1}
    &\quad&\by{\cref{eq:Zi_squared_sum}}\\
    &\leq
    \Lest 
    \sqrt{\frac{12 S_{k-1}}{k^3}}
    +
    \frac{\Lest \Mtrue}{2 \Mest_k} 
    \sqrt{\frac{S_{k-1}}{k^3}}
    &\quad&\by{\cref{eq:condition_continue_k-1}}\\
    &\leq
    \frac{\Lest \Mmax}{\Mest_{k-1}}
    \sqrt{\frac{12 S_{k-1}}{k^3}}
    +
    \frac{\Lest \Mmax}{2 \Mest_{k-1}} 
    \sqrt{\frac{S_{k-1}}{k^3}}
    &\quad&\by{$\Mest_{k-1}, \Mtrue \leq \Mmax$}\\
    &\leq
    \frac{4 \Lest \Mmax}{\Mest_{k-1}}
    \sqrt{\frac{S_{k-1}}{k^3}},
  \end{alignat}
  which is the first inequality of the desired result~\cref{eq:grad_norm_ybar_upperbound}.
  The second inequality follows from \cref{eq:condition_continue_k-1}:
  \begin{align}
    \frac{4 \Lest \Mmax}{\Mest_{k-1}}
    \sqrt{\frac{S_{k-1}}{k^3}}
    \leq
    \frac{4 \Lest^2 \Mmax}{M_{k-1}^2 k^4}
    \leq
    \frac{4 \Lmax^2 \Mmax}{M_0 k^4}.
  \end{align}
\end{proof}
\begin{proof}[Proof of \cref{eq:grad_norm_ybar_upperbound_tight}]
  When $M_k$ is computed with \cref{eq:update_Mest}, we have \cref{eq:grad_ybark_upperbound_Mestk}.
  Similarly to the proof of \cref{eq:grad_norm_ybar_upperbound}, using \cref{eq:grad_ybark_upperbound_Mestk} instead of \cref{eq:grad_ybark_upperbound_Mf} yields
  \begin{alignat}{2}
    \min_{1 \leq i < k} \norm*{ \nabla f (\bar y_i) }
    &\leq
    \Lest 
    \sqrt{\frac{12 S_{k-1}}{k^3}}
    +
    \frac{k}{2} \Mest_{k-1} S_{k-1}.
  \end{alignat}
  The right-hand side is bounded as
  \begin{alignat}{2}
    \Lest 
    \sqrt{\frac{12 S_{k-1}}{k^3}}
    +
    \frac{k}{2} \Mest_{k-1} S_{k-1}
    &\leq
    \Lest 
    \sqrt{\frac{12 S_{k-1}}{k^3}}
    +
    \frac{\Lest}{2} \sqrt{\frac{S_{k-1}}{k^3}}
    &\quad&\by{\cref{eq:condition_continue_k-1}}\\
    &\leq
    4 \Lest
    \sqrt{\frac{S_{k-1}}{k^3}},
  \end{alignat}
  which implies the first inequality in \cref{eq:grad_norm_ybar_upperbound_tight}.
  The second inequality is obtained similarly to the second inequality of \cref{eq:grad_norm_ybar_upperbound}.
\end{proof}

\subsection{Main results}
\label{sec:complexity_bound}
Combining \cref{lem:grad_norm_ybar_upperbound} with conditions \cref{eq:condition_continue,eq:descent_condition}, we obtain the following complexity bound for \cref{alg:proposed_agd}.
\begin{theorem}
  \label{thm:complexity}
  Suppose that \cref{asm:gradient_hessian_lip} holds and let $\Lmax \coloneqq \max\set{\Linit, \alpha \Ltrue}$, $\Mmax \coloneqq \max\set{M_0, \Mtrue}$, $\Delta \coloneqq f(x_\mathrm{init}) - \inf_{x \in \R^d} f(x)$,
  \begin{align}
    c_1
    \coloneqq
    \log_\alpha \prn*{\frac{1}{\beta}},
    \quad\text{and}\quad
    c_2
    \coloneqq
    \log_\alpha \prn*{\frac{\Lmax}{\Linit}}.
    \label{eq:def_c1_c2}
  \end{align}
  In \cref{alg:proposed_agd}, when $\bar y_k$ defined by \cref{eq:def_ybark} satisfies $\norm*{\nabla f(\bar y_k)} \leq \epsilon$ for the first time, the total iteration count $K$ is at most
  \begin{align}
    \prn[\bigg]{
      70
      + 
      108 c_1 \frac{\Mmax^{1/4}}{M_0^{1/4}}
    }
    \frac{\Lmax^{1/2} \Mmax^2 \Delta}{\epsilon^{7/4} M_0^{7/4}}
    + 2 (c_2 + 1) \frac{\Lmax^{1/2} \Mmax^{1/4}}{\epsilon^{1/4} M_0^{1/2}}.
    \label{eq:complexity_bound}
  \end{align}
  If $M_k$ is computed with \cref{eq:update_Mest} in the algorithm and $\beta$ is set as $\beta = 1$, the above bound is improved to
  \begin{align}
    70
    \frac{\Lmax^{1/2} \Mmax^{1/4} \Delta}{\epsilon^{7/4}}
    +
    2 (c_2 + 1) \frac{\Lmax^{1/2}}{M_0^{1/4} \epsilon^{1/4}}.
    \label{eq:complexity_bound_tight}
  \end{align}
\end{theorem}
\begin{proof}[Proof of bound \cref{eq:complexity_bound}]
  We count the number of iterations of \cref{alg:proposed_agd} separately for three types of epoch:
  \begin{itemize}
    \item 
    successful epoch: an epoch that does not find an $\epsilon$-stationary point and ends at Line \ref{alg-line-agd:restart_successful} with the descent condition~\cref{eq:descent_condition} satisfied,
    \item
    unsuccessful epoch: an epoch that does not find an $\epsilon$-stationary point and ends at Line \ref{alg-line-agd:restart_unsuccessful} with the descent condition~\cref{eq:descent_condition} unsatisfied,
    \item
    last epoch: the epoch that finds an $\epsilon$-stationary point.
  \end{itemize}

  \paragraph{Successful epochs}
  Let us focus on a successful epoch.
  Let $k$ be the iteration number of the epoch.
  It follows that 
  \begin{align}
    (k+1)^5
    \Mest_k^2 S_k
    &>
    \Lest^2.
    \label{eq:restart_condition_k}
  \end{align}
  When $k \geq 2$, we have
  \begin{alignat}{2}
    \epsilon
    &<
    \min_{1 \leq i < k} \norm*{ \nabla f (\bar y_i) }
    &\quad&\text{(since this epoch does not find an $\epsilon$-stationary point)}\\
    &\leq
    \frac{4 \Lest \Mmax}{\Mest_{k-1}}
    \sqrt{\frac{S_{k-1}}{k^3}}
    &\quad&\by{\cref{eq:grad_norm_ybar_upperbound}}\\
    &\leq
    \frac{4 \Lest \Mmax}{\Mest_0} \sqrt{\frac{S_k}{k^3}}
    &\quad&\by{$M_{k-1} \geq M_0$ and $S_k \geq S_{k-1}$},
  \end{alignat}
  and hence 
  \begin{align}
    S_k > \prn*{\frac{\epsilon \Mest_0}{4 \Lest \Mmax}}^2 k^3.
    \label{eq:Sk_lowerbound}
  \end{align}
  When $k = 1$, we also have \cref{eq:Sk_lowerbound}, as follows:
  \begin{align}
    S_k
    =
    \norm*{x_1 - x_0}^2
    =
    \frac{1}{\Lest^2} \norm*{\nabla f(x_0)}^2
    =
    \frac{1}{\Lest^2} \norm*{\nabla f(\bar y_1)}^2
    >
    \frac{\epsilon^2}{\Lest^2}
    \geq
    \prn*{\frac{\epsilon \Mest_0}{4 \Lest \Mmax}}^2 k^3.
  \end{align}
  Combining \cref{eq:Sk_lowerbound,eq:restart_condition_k} yields
  \begin{align}
    S_k
    =
    S_k^{7/8}
    S_k^{1/8}
    &\geq
    \prn[\bigg]{\prn*{\frac{\epsilon \Mest_0}{4 \Lest \Mmax}}^2 k^3}^{7/8}
    \prn*{\frac{\Lest^2}{\Mest_k^2 (k+1)^5}}^{1/8}\\
    &=
    \frac{(\epsilon / 4)^{7/4}}{\Lest^{3/2}}
    \frac{M_0^{7/4}}{\Mmax^{7/4} \Mest_k^{1/4}}
    \frac{k^{21/8}}{(k+1)^{5/8}}
    \geq
    \frac{(\epsilon / 4)^{7/4}}{\Lest^{3/2}}
    \frac{M_0^{7/4}}{\Mmax^2}
    \frac{k^{21/8}}{(k+1)^{5/8}}.
  \end{align}
  Plugging this bound into condition \cref{eq:descent_condition}, we obtain
  \begin{align}
    f(x_0) - f(x_k)
    \geq
    \frac{\Lest S_k}{2 (k+1)}
    &\geq
    \frac{(\epsilon / 4)^{7/4}}{\Lest^{1/2}}
    \frac{M_0^{7/4}}{\Mmax^2}
    \prn*{\frac{k}{k+1}}^{13/8}
    \frac{k}{2}\\
    &\geq
    \frac{(\epsilon / 4)^{7/4}}{\Lmax^{1/2}}
    \frac{M_0^{7/4}}{\Mmax^2}
    \prn*{\frac{1}{2}}^{13/8}
    \frac{k}{2}
    \geq
    \frac{\epsilon^{7/4} M_0^{7/4}}{\Lmax^{1/2} \Mmax^2}
    \frac{k}{70},
  \end{align}
  and thus
  \begin{align}
    k
    \leq
    70
    \frac{\Lmax^{1/2} \Mmax^2}{\epsilon^{7/4} M_0^{7/4}}
    \prn*{f(x_0) - f(x_k)}.
  \end{align}
  Summing this bound over all successful epochs leads to the conclusion that the total iteration number of successful epochs is at most
  \begin{align}
    70 \frac{\Lmax^{1/2} \Mmax^2 \Delta}{\epsilon^{7/4} M_0^{7/4}}.
    \label{eq:Ksuc_upperbound}
  \end{align}
  For later use, we will also evaluate the number of successful epochs.
  Combining \cref{eq:Sk_lowerbound,eq:restart_condition_k} also yields
  \begin{align}
    S_k
    =
    S_k^{3/4}
    S_k^{1/4}
    &\geq
    \prn[\bigg]{ \prn*{\frac{\epsilon \Mest_0}{4 \Lest \Mmax}}^2 k^3 }^{3/4}
    \prn*{\frac{\Lest^2}{\Mest_k^2 (k+1)^5}}^{1/4}\\
    &=
    \frac{(\epsilon / 4)^{3/2}}{\Lest}
    \frac{\Mest_0^{3/2}}{\Mmax^{3/2} \Mest_k^{1/2}}
    \frac{k^{9/4}}{(k+1)^{5/4}}
    \geq
    \frac{(\epsilon / 4)^{3/2}}{\Lest}
    \frac{\Mest_0^{3/2}}{\Mmax^2}
    \frac{k^{9/4}}{(k+1)^{5/4}}.
  \end{align}
  Plugging this bound into condition \cref{eq:descent_condition} gives
  \begin{align}
    f(x_0) - f(x_k)
    \geq
    \frac{\Lest S_k}{2 (k+1)}
    &\geq
    \frac{(\epsilon / 4)^{3/2}}{2}
    \frac{\Mest_0^{3/2}}{\Mmax^2}
    \prn*{\frac{k}{k+1}}^{9/4}\\
    &\geq
    \frac{(\epsilon / 4)^{3/2}}{2}
    \frac{\Mest_0^{3/2}}{\Mmax^2}
    \prn*{\frac{1}{2}}^{9/4}
    =
    \frac{\epsilon^{3/2} \Mest_0^{3/2}}{2^{25/4} \Mmax^2}.
  \end{align}
  Thus, we can deduce that the number $N_{\mathrm{suc}}$ of successful epochs satisfies
  \begin{align}
    N_{\mathrm{suc}}
    \leq
    2^{25/4}
    \frac{\Mmax^2 \Delta}{\epsilon^{3/2} \Mest_0^{3/2}}.
    \label{eq:Nsuc_upperbound}
  \end{align}

  \paragraph{Other epochs}
  We deduce from \cref{eq:grad_norm_ybar_upperbound} that the number of iterations for each epoch is at most
  \begin{align}
    \prn*{
      \frac{4 \Lmax^2 \Mmax}{\epsilon M_0^2}
    }^{1/4}.
  \end{align}
  Let $N_{\mathrm{unsuc}}$ be the number of unsuccessful epochs.
  We then have
  \begin{align}
    \Linit
    \beta^{N_{\mathrm{suc}}}
    \alpha^{N_{\mathrm{unsuc}}}
    \leq
    \Lmax,
  \end{align}
  and hence
  \begin{align}
    N_{\mathrm{unsuc}}
    \leq
    c_1 N_{\mathrm{suc}} + c_2,
    \label{eq:Nunsuc_upperbound}
  \end{align}
  where $c_1$ and $c_2$ are defined by \cref{eq:def_c1_c2}.
  The total iteration number of unsuccessful and last epochs is at most
  \begin{alignat}{2}
    &\mathInd
    (N_{\mathrm{unsuc}} + 1)
    \prn*{ \frac{4 \Lmax^2 \Mmax}{\epsilon M_0^2} }^{1/4}\\
    &\leq
    \prn*{ c_1 N_{\mathrm{suc}} + c_2 + 1 }
    \prn*{ \frac{4 \Lmax^2 \Mmax}{\epsilon M_0^2} }^{1/4}
    &\quad&\by{\cref{eq:Nunsuc_upperbound}}\\
    &\leq
    2^{27/4} c_1
    \frac{\Lmax^{1/2} \Mmax^{9/4} \Delta}{\epsilon^{7/4} \Mest_0^2}
    + 2^{1/2} (c_2 + 1) \frac{\Lmax^{1/2} \Mmax^{1/4}}{\epsilon^{1/4} M_0^{1/2}}
    &\quad&\by{\cref{eq:Nsuc_upperbound}}\\
    &\leq
    108 c_1
    \frac{\Lmax^{1/2} \Mmax^{9/4} \Delta}{\epsilon^{7/4} \Mest_0^2}
    + 2 (c_2 + 1) \frac{\Lmax^{1/2} \Mmax^{1/4}}{\epsilon^{1/4} M_0^{1/2}}.
  \end{alignat}
  Putting this bound together with \cref{eq:Ksuc_upperbound} concludes the proof.
\end{proof}

\begin{proof}[Proof of bound \cref{eq:complexity_bound_tight}]
  The proof is similar to \cref{eq:complexity_bound} through the evaluation of the total iteration number of successful, unsuccessful, and last epochs.

  \paragraph{Successful epochs}
  Let us focus on a successful epoch.
  Let $k$ be the iteration number of the epoch.
  As a counterpart of \cref{eq:Sk_lowerbound}, we can obtain 
  \begin{align}
    S_k > \prn*{\frac{\epsilon}{4 \Lest}}^2 k^3
    \label{eq:Sk_lowerbound_tight}
  \end{align}
  from \cref{eq:grad_norm_ybar_upperbound_tight}.
  Combining \cref{eq:Sk_lowerbound_tight,eq:restart_condition_k} yields
  \begin{align}
    S_k
    &=
    S_k^{7/8}
    S_k^{1/8}
    \geq
    \prn*{\prn*{\frac{\epsilon}{4 \Lest}}^2 k^3}^{7/8}
    \prn*{\frac{\Lest^2}{\Mest_k^2 (k+1)^5}}^{1/8}
    =
    \frac{(\epsilon / 4)^{7/4}}{\Lest^{3/2} \Mest_k^{1/4}}
    \frac{k^{21/8}}{(k+1)^{5/8}}.
  \end{align}
  Plugging this bound into condition \cref{eq:descent_condition}, we obtain
  \begin{align}
    f(x_0) - f(x_k)
    \geq
    \frac{\Lest S_k}{2 (k+1)}
    &\geq
    \frac{(\epsilon / 4)^{7/4}}{\Lest^{1/2} \Mest_k^{1/4}}
    \prn*{\frac{k}{k+1}}^{13/8}
    \frac{k}{2}\\
    &\geq
    \frac{(\epsilon / 4)^{7/4}}{\Lmax^{1/2} \Mmax^{1/4}}
    \prn*{\frac{1}{2}}^{13/8}
    \frac{k}{2}
    \geq
    \frac{\epsilon^{7/4}}{\Lmax^{1/2} \Mmax^{1/4}}
    \frac{k}{70}.
  \end{align}
  Hence, as in the proof of \cref{eq:complexity_bound}, the total iteration number of all successful epochs is at most
  \begin{align}
    70
    \frac{\Lmax^{1/2} \Mmax^{1/4} \Delta}{\epsilon^{7/4}}.
    \label{eq:Ksuc_upperbound_tight}
  \end{align}

  \paragraph{Other epochs}
  We deduce from \cref{eq:grad_norm_ybar_upperbound_tight} that the iteration number of each epoch is at most $\prn[\big]{ \frac{4 \Lmax^2}{M_0 \epsilon} }^{1/4}$.
  In addition, unsuccessful epochs occur only at most $c_2$ times when $\beta = 1$, where $c_2$ is defined by \cref{eq:def_c1_c2}.
  Thus, the total iteration number of unsuccessful and last epochs is at most
  \begin{align}
    \prn*{
      c_2 + 1
    }
    \prn*{
      \frac{4 \Lmax^2}{M_0 \epsilon}
    }^{1/4}
    \leq
    2 (c_2 + 1) \frac{\Lmax^{1/2}}{M_0^{1/4} \epsilon^{1/4}}.
  \end{align}
  Putting this bound together with \cref{eq:Ksuc_upperbound_tight} concludes the proof.
\end{proof}

Note that bound \cref{eq:complexity_bound_tight} is better than \cref{eq:complexity_bound} since $M_0 \leq \Mmax$.
If the input parameters of \cref{alg:proposed_agd} are set as $\Linit = \Theta(L_f)$ and $M_0 = \Theta(M_f)$, both bounds \cref{eq:complexity_bound,eq:complexity_bound_tight} are simplified to $O(L_f^{1/2} M_f^{1/4} \Delta \epsilon^{-7/4})$.
If $\Linit$ and $M_0$ deviate significantly from $L_f$ and $M_f$, the coefficient of the $O(\epsilon^{-7/4})$ or $O(\epsilon^{-1/4})$ terms in the bounds will be large.
We emphasize, however, that \cref{thm:complexity} is valid no matter what values are chosen for the parameters $\Linit, M_0 > 0$, $\alpha > 1$, and $0 < \beta \leq 1$, which shows the algorithm's robustness.

The number of function and gradient evaluations is of the same order as the iteration complexity given in \cref{thm:complexity}, because \cref{alg:proposed_agd} evaluates the objective function and the gradient at two or three points in each iteration.


Since \cref{alg:proposed_agd} does not need the target accuracy $\epsilon$ as input, we can deduce the following global convergence property as a byproduct of the complexity analysis.
\begin{corollary}
  \label{cor:global_convergence}
  Suppose that \cref{asm:gradient_hessian_lip} holds.
  Let $z_K$ denote $\bar y_k$ at the total iteration $K$ of \cref{alg:proposed_agd}.
  Then, the following holds:
  \begin{align}
    \lim_{K \to \infty}
    \min_{1 \leq i \leq K}
    \norm*{ \nabla f(z_i) }
    = 0.
  \end{align}
\end{corollary}

\subsection{Discussion}
\label{sec:discussion}
To gain more insight into the complexity analysis provided above, let us discuss it from three perspectives.

\paragraph{Acceleration parameter}
In order to deduce from \cref{eq:grad_ybark_upperbound_Mestk} that the gradient norm $\norm*{ \nabla f (\bar y_k) }$ is small, the $Z_k$ should be large, and hence the acceleration parameter $\theta_k$ should also be large, from definition \cref{eq:def_ybark} of $Z_k$.
At the same time, $\theta_k$ should be small to obtain a significant objective decrease from \cref{eq:descent_condition}.
Our choice of $\theta_k = \frac{k}{k+1}$ in \cref{eq:def_theta_k} strikes this balance, and the essence is the following two conditions: $\theta_k = 1 - \Theta(k^{-1})$ and $\theta_{k+1}^2 \leq \theta_k \leq \theta_{k+1}$.
As long as $\theta_k$ satisfies them, we can ensure a complexity bound of $O(\epsilon^{-7/4})$ even with a different choice than in \cref{eq:def_theta_k}.

\paragraph{Restart}
The proposed AGD restarts when $S_k = \sum_{i=1}^k \norm*{x_i - x_{i-1}}^2$ is large, which plays two roles in our analysis.
First, as long as the epoch continues, the solutions $y_0,\dots,y_{k-1}$ are guaranteed to be close together, and therefore \cref{lem:gradient_jensen} gives a good approximation $\sum_{i=0}^{k-1} p_{k,i} \nabla f \prn*{ y_i }$ of the gradient $\nabla f(\bar y_k)$.
Second, the restart mechanism ensures that $\norm*{x_k - x_{k-1}}^3$ and $\norm*{x_k - x_{k-1}}^4$ are small compared to $\norm*{x_k - x_{k-1}}^2$ during the epoch, which helps to derive \cref{lem:decrease_epoch} from \cref{lem:potential_decrease_iteration}.




\paragraph{Potential}
Jin et~al.~\cite{jin2018accelerated} also use a potential function to analyze AGD under \cref{asm:gradient_hessian_lip}, but our proof technique differs significantly from theirs.
First, our potential~\cref{eq:def_potential} is more complicated than the existing one.
This complication is mainly caused by the dependence of our $\theta_k$ on $k$; $\theta_k$ in \cite{jin2018accelerated} does not depend on $k$, while our $\theta_k$ does but with the benefit of not requiring knowledge of $M_f$ or $\epsilon$.
Second, Jin et~al.~\cite{jin2018accelerated} prove the potential decrease only in regions where $f$ is weakly convex and use a different technique, negative curvature descent, in other regions.
In contrast, our \cref{lem:potential_decrease_iteration} is valid in all regions and provides a unified analysis with the potential function.
This unified analysis is made possible by taking full advantage of the Lipschitz continuous Hessian through \cref{lem:gradient_jensen,lem:trapezoidal_rule_error}.

\section{Numerical experiments}
This section compares the performance of the proposed method and six existing methods.
We implemented objective functions in Python with JAX~\cite{jax2018github} and Flax~\cite{flax2020github}.
We implemented some methods in Python and used SciPy implementation~\cite{virtanen2020scipy} for other methods.
They were executed on a computer with an Apple M1 Chip (8 cores, 3.2 GHz) and 16 GB RAM.
The source code used in the experiments is available on GitHub.%
\footnote{
  \url{https://github.com/n-marumo/restarted-agd}
}

\subsection{Compared algorithms}
We compared the following algorithms: \Proposed, \GD, \JNJ, \LL, \OC, \LBFGS, and \CG.
\begin{itemize}
  \item 
  \Proposed is \cref{alg:proposed_agd} with parameters set as $(\Linit, M_0, \alpha, \beta) = (10^{-3}, 10^{-16},\allowbreak 2, 0.9)$.
  We used \cref{eq:update_Mest_practical} for updating $M_k$.
  \item 
  \GD is gradient descent with Armijo-type backtracking.
  This method has input parameters $\Linit$, $\alpha$, and $\beta$ similar to those in \Proposed, which were set as $(\Linit, \alpha, \beta) = (10^{-3}, 2, 0.9)$.
  \item
  \JNJ \cite[Algorithm~2]{jin2018accelerated} is an AGD method reviewed in \cref{sec:related_work}.
  The parameters were set in accordance with \cite[Eq.~(3)]{jin2018accelerated}.
  The equation involves constants $c$ and $\chi$, whose values are difficult to determine; we set them as $c = \chi = 1$.
  \item
  \LL \cite[Algorithm~2]{li2022restarted} is a state-of-the-art AGD method that makes use of the practical superiority of AGD.
  The parameters were set in accordance with \cite[Theorem~2.2 and Section~4]{li2022restarted}.
  \item
  \OC is an AGD method with an adaptive restart scheme proposed in \cite{odonoghue2015adaptive} for \emph{convex} optimization.
  We used the gradient scheme among the two restart schemes proposed in \cite[Section~3.2]{odonoghue2015adaptive}.
  The parameter of \cite[Algorithm~1]{odonoghue2015adaptive} was set as $q = 0$, and the step size was determined by backtracking as in \GD.
  \item
  \LBFGS is the limited-memory BFGS method~\cite{byrd1995limited}.
  We used the SciPy implementation \cite{virtanen2020scipy}, i.e., \texttt{scipy.optimize.minimize} with option \texttt{method="L-BFGS-B"}.
  \item
  \CG is a nonlinear conjugate gradient algorithm, a variant of the Fletcher--Reeves method described in \cite[pp.120--122]{nocedal2006numerical}.
  We used the SciPy implementation \cite{virtanen2020scipy}, i.e., \texttt{scipy.optimize.minimize} with option \texttt{method="CG"}.
\end{itemize}
\GD has a complexity bound of $O(\epsilon^{-2})$, \JNJ does $\tilde O(\epsilon^{-7/4})$, and \LL and \Proposed do $O(\epsilon^{-7/4})$.
The other methods, \OC, \LBFGS, and \CG, do not have a guarantee superior to $O(\epsilon^{-2})$ for general nonconvex problems.

The parameter setting for \JNJ and \LL requires the values of the Lipschitz constants $L_f$ and $M_f$ and the target accuracy $\epsilon$.
For these two methods, we tuned the best $L_f$ among $\set{10^{-4},10^{-3},\dots,10^4}$ and set $M_f = 1$ and $\epsilon = 10^{-16}$ following \cite{li2022restarted}.
Note that if these values deviate from the actual values, the methods do not guarantee convergence.

\subsection{Problem setting}
We test the performance of the algorithms with three types of problem instances.

The first two instances are of training neural networks for classification and autoencoder:
\begin{align}
  \min_{w \in \R^d} \ 
  &
  \frac{1}{N}
  \sum_{i=1}^N
  \ell_{\mathrm{CE}}(y_i, \phi_1(x_i; w)),
  \label{eq:instance_classification}\\
  \min_{w \in \R^d} \ 
  &
  \frac{1}{2MN}
  \sum_{i=1}^N
  \norm*{x_i - \phi_2(x_i; w)}^2.
  \label{eq:instance_autoencoder}
\end{align}
Above, $x_1,\dots,x_N \in \R^M$ are data vectors, $y_1,\dots,y_N \in \set{0, 1}^K$ are one-hot label vectors, $\ell_{\mathrm{CE}}$ is the cross-entropy loss, and $\phi_1(\cdot; w): \R^M \to \R^K$ and $\phi_2(\cdot; w): \R^M \to \R^M$ are neural networks parameterized by $w \in \R^d$.
We used $N = 10000$ randomly chosen data from the MNIST handwritten digit dataset.
The dataset consists of $K$-class images, each represented as an $M$-dimensional vector, where $K = 10$ and $M = 784$.
The number of parameters was $d = 25818$ in \cref{eq:instance_classification} and $d = 52064$ in \cref{eq:instance_autoencoder}.
See \cref{sec:detail_neural_net} for the details of the neural networks used in the experiments.

The third instance is low-rank matrix completion:
\begin{align}
  \min_{\substack{U \in \R^{p \times r}\\V \in \R^{q \times r}}}\ 
  \frac{1}{2 N}
  \sum_{(i, j, s) \in \Omega}
  \prn*{(U V^\top)_{ij} - s}^2
  + \frac{1}{2 N} \norm*{U^\top U - V^\top V}_{\mathrm{F}}^2.
\end{align}
The set $\Omega$ consists of $N$ observed entries of a $p \times q$ data matrix, and $(i, j, s) \in \Omega$ means that the $(i, j)$-entry is $s$.
The partially observed matrix is approximated by the matrix $U V^\top$ of rank $r$.
The regularization term with the Frobenius norm $\norm{\cdot}_{\mathrm{F}}$ was proposed in \cite{tu2016low} as a way to balance $U$ and $V$.
We used the MovieLens-100K dataset \cite{harper2015movielens}, where $p = 943$, $q = 1682$, and $N = 100000$.
We set the rank as $r \in \set{100, 200}$.
The number of parameters was $(p+q)r \in \set{262500, 525000}$.

\subsection{Results}
\subsubsection{Behavior of proposed algorithm}
\Cref{fig:experiments_objlm} shows the objective function value $f(x_k)$ and the estimates $\Lest$ and $\Mest_k$ at each iteration of \Proposed.
The iterations at which a restart occurred are also marked; ``successful'' and ``unsuccessful'' mean restarts at Line \ref{alg-line-agd:restart_successful} and Line \ref{alg-line-agd:restart_unsuccessful} of \cref{alg:proposed_agd}, respectively.

This figure illustrates that the proposed algorithm adaptively estimates the Lipschitz constants $\Ltrue$ and $\Mtrue$.
In particular, in \cref{fig:exp_objlm_classification}, the values of $\Lest$ and $\Mest_k$ differ significantly between the early and final stages; the algorithm adopts a larger step size (i.e., smaller $\Lest$) near the stationary point.
This adaptability is expected to improve performance over algorithms that treat the Lipschitz constants as given values.

We can also see that the proposed method restarts frequently in the early stages but that the frequency decreases as the iterations progress.
This behavior is explained by the restart condition on Line~\ref{alg-line-agd:check_large_Sk} of \cref{alg:proposed_agd}; as the algorithm progresses, the iterate $x_k$ moves less in general, i.e., $S_k \coloneqq \sum_{i=1}^k \norm*{x_i - x_{i-1}}^2$ grows slower, and the restart condition holds less frequently.

The value of $\Mest_k$ becomes very small at each restart because we set $\Mest_0 = 10^{-16}$.
The value is updated immediately, so changing $\Mest_0$ to $10^{-12}$ or $0$, for example, does not affect the algorithm's behavior for the problem instances.

\begin{figure}[!t]
  \centering
  \includegraphics[height=3.5ex]{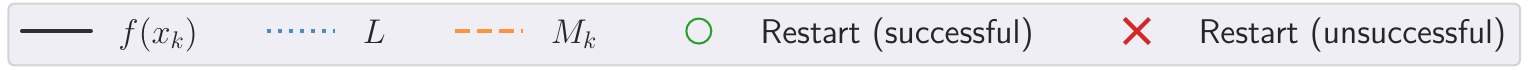}\par%
  \subfloat[Classification\label{fig:exp_objlm_classification}]{%
    \includegraphics[width=0.46\linewidth]{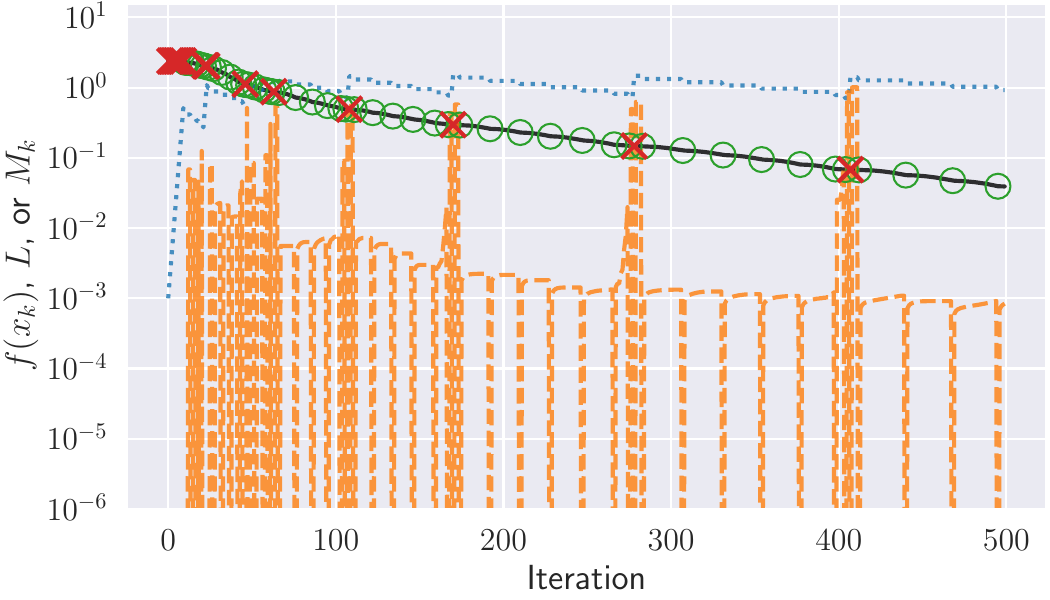}\hspace{0.05\linewidth}%
    \includegraphics[width=0.46\linewidth]{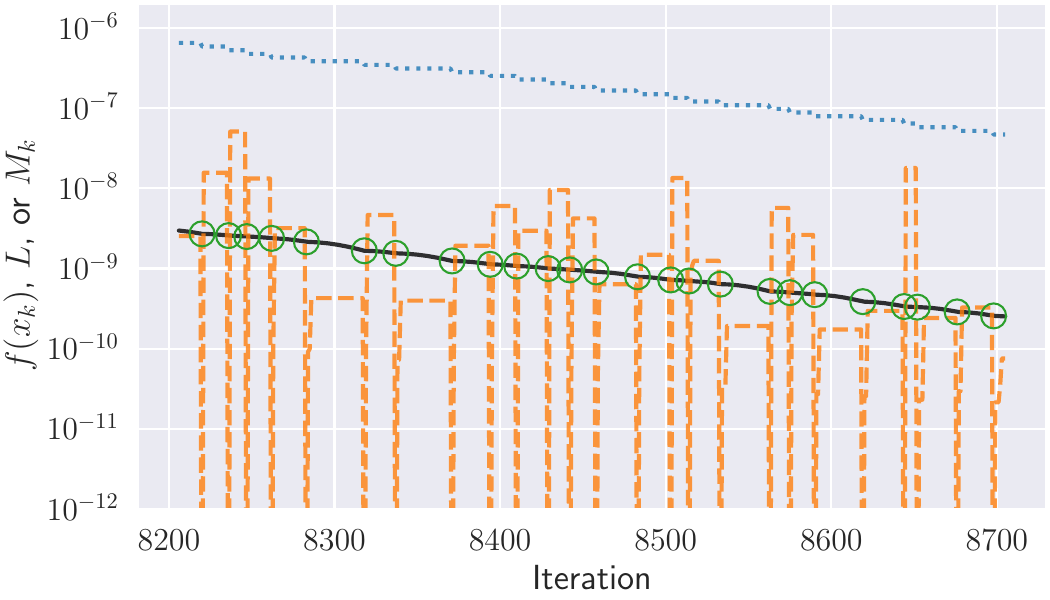}%
  }\par%
  \subfloat[Autoencoder\label{fig:exp_objlm_autoencoder}]{%
    \includegraphics[width=0.46\linewidth]{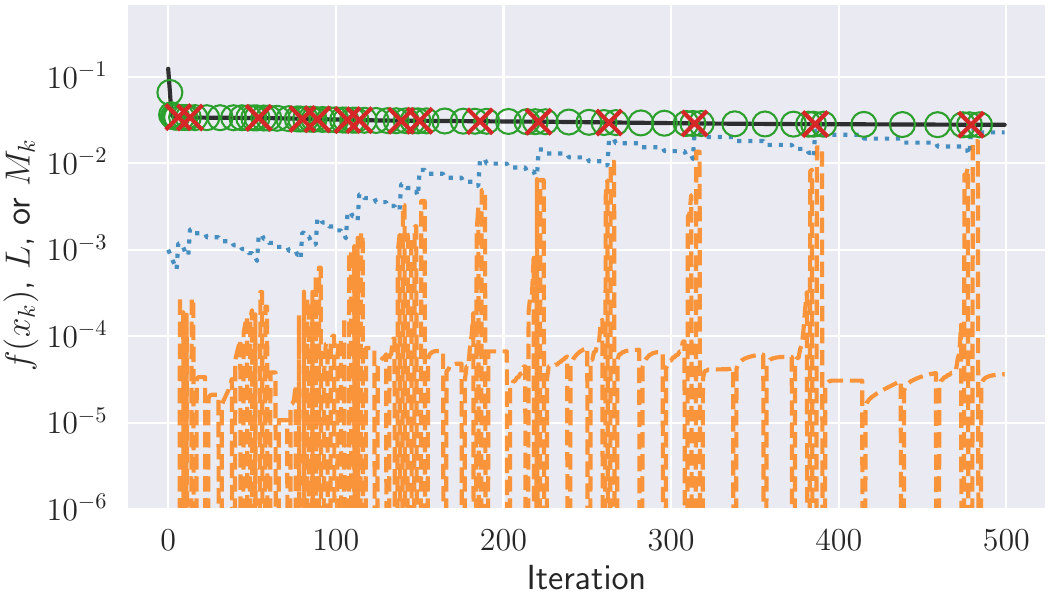}\hspace{0.05\linewidth}%
    \includegraphics[width=0.46\linewidth]{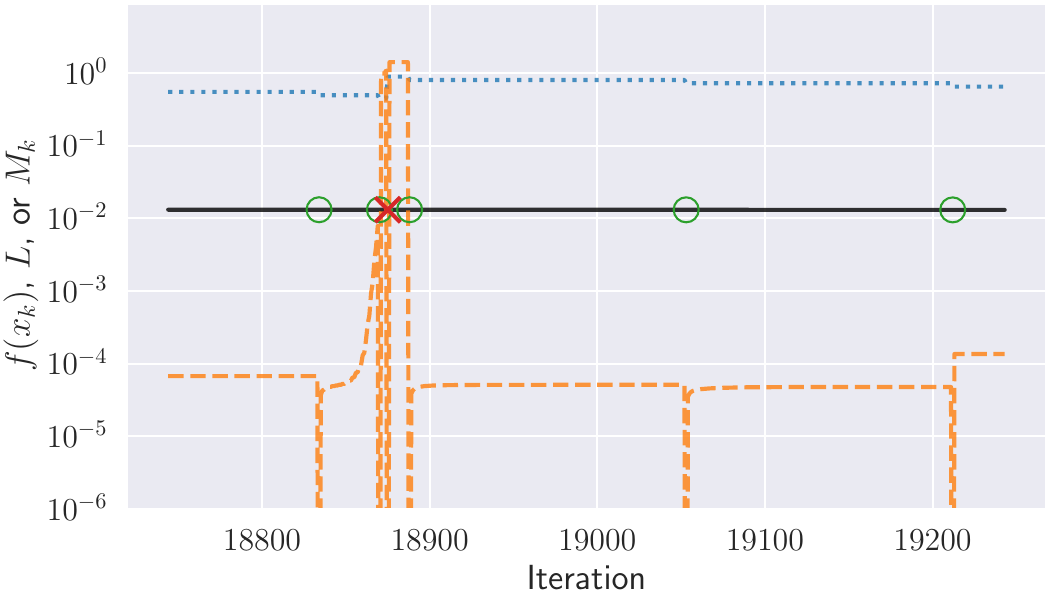}%
  }\par%
  \subfloat[Matrix completion ($r = 100$)\label{fig:exp_objlm_completion_100}]{%
    \includegraphics[width=0.46\linewidth]{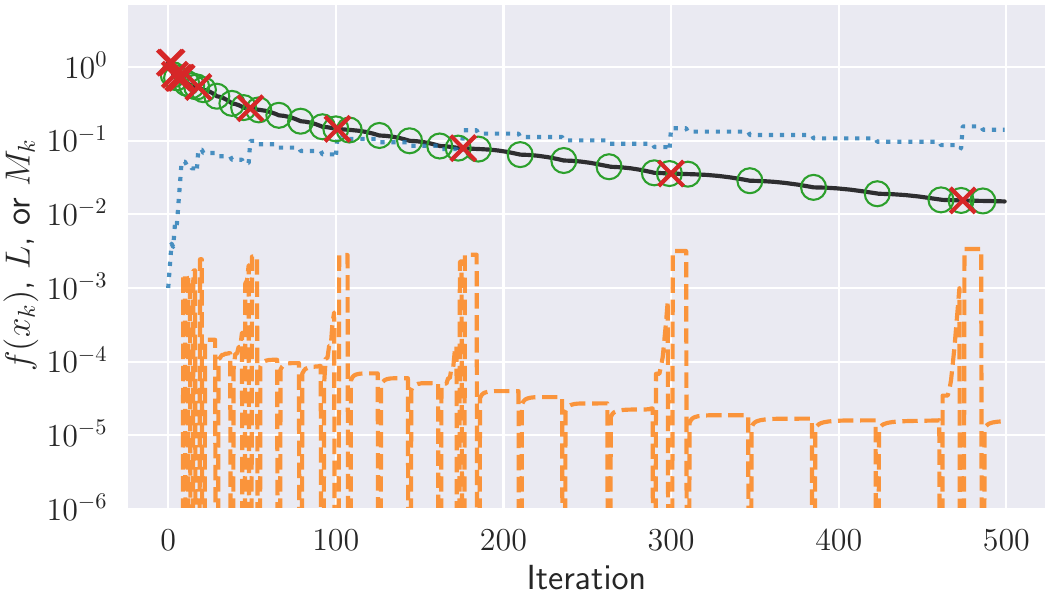}\hspace{0.05\linewidth}%
    \includegraphics[width=0.46\linewidth]{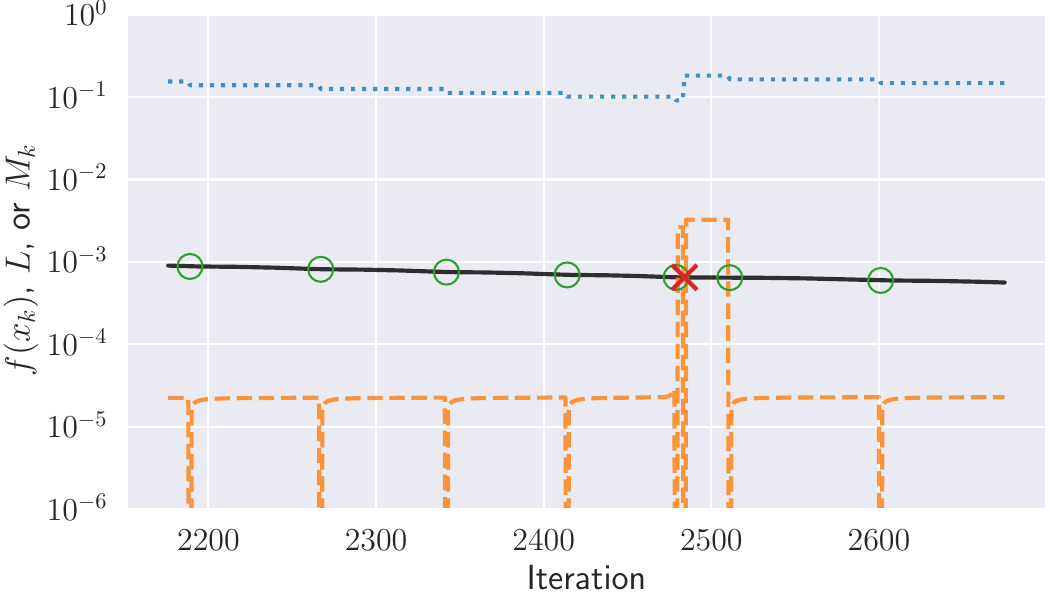}%
  }\par%
  \subfloat[Matrix completion ($r = 200$)\label{fig:exp_objlm_completion_200}]{%
    \includegraphics[width=0.46\linewidth]{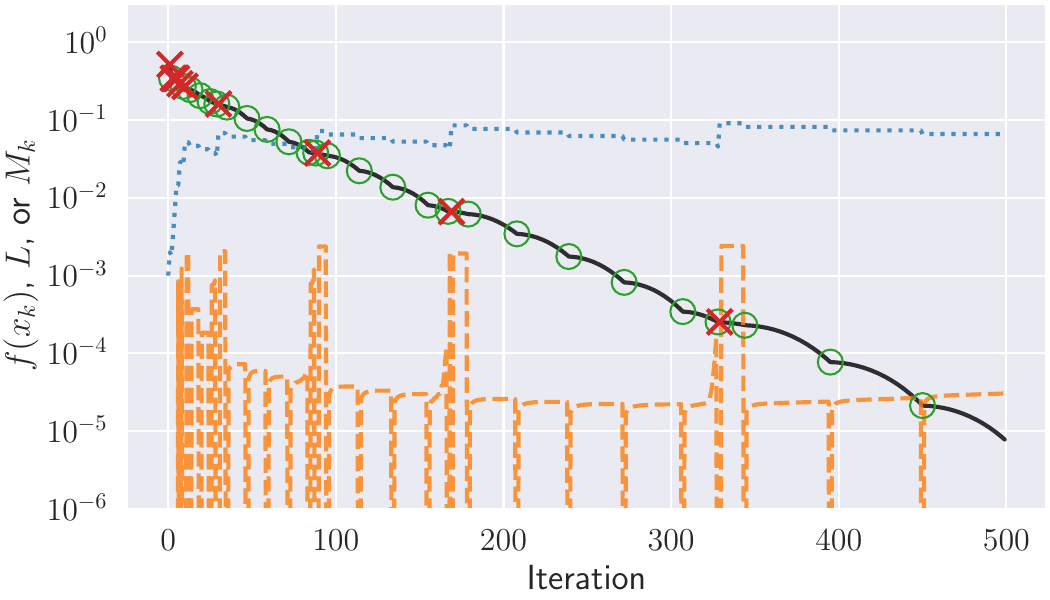}\hspace{0.05\linewidth}%
    \includegraphics[width=0.46\linewidth]{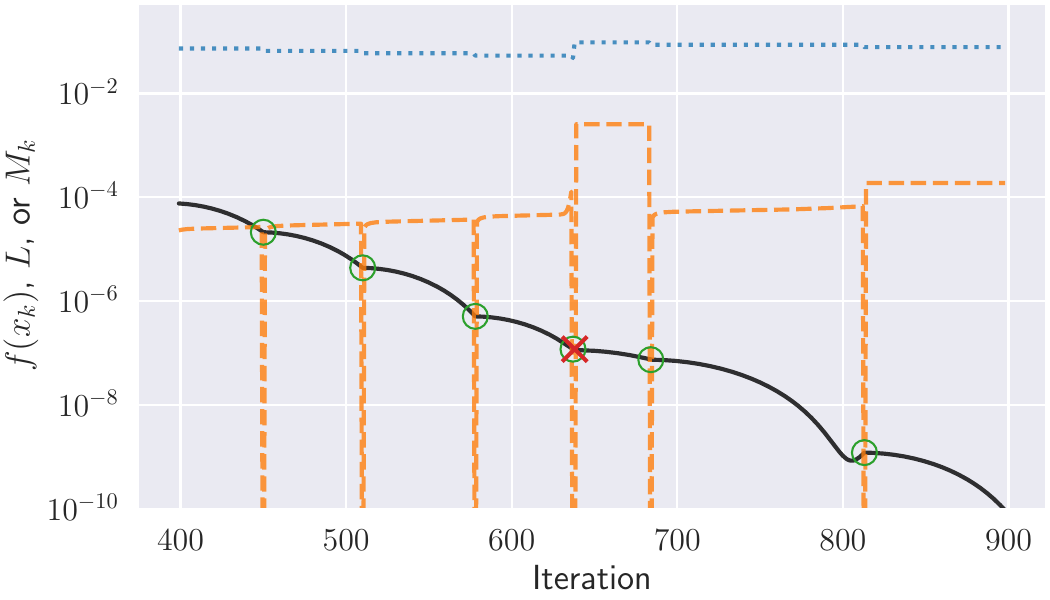}%
  }\par%
  \caption{
    The objective function value $f(x_k)$ and the estimates $L$ and $M_k$ at each iteration of the proposed method.
    The iterations at which a restart occurred are marked.
    Left: the first 500 iterations.
    Right: 500 iterations near a stationary point.
    \label{fig:experiments_objlm}
  }
\end{figure}

\subsubsection{Comparison with algorithms with $O(\epsilon^{-7/4})$ or $\tilde O(\epsilon^{-7/4})$ guarantees}
\cref{fig:experiments} shows the objective function values and the gradient norm obtained and compares \Proposed with \JNJ and \LL, which have complexity bounds of $O(\epsilon^{-7/4})$ or $\tilde O(\epsilon^{-7/4})$.
\Proposed generally performs well, especially in reducing the gradient norm.

\Cref{fig:exp_classification,fig:exp_completion_200} show that \Proposed converges faster than the sublinear rate guaranteed by the theoretical analysis, $O(\epsilon^{-7/4})$.
This speed-up is because \Proposed automatically estimates the Lipschitz constants and successfully captures the local curvature of $f$.
In the setting shown in \cref{fig:exp_completion_100}, \LL reduces the objective function value faster than \Proposed, probably because the two methods converge to different stationary points; \LL is fortunate in this instance to have found a better stationary point in terms of the function value.

We emphasize again that we fixed all of the input parameters for \Proposed, whereas we tuned $L_f$ for the input parameter for \LL and \JNJ.
\Proposed achieved comparable or better performance than those of the (theoretically) state-of-the-art methods without parameter tuning.

\begin{figure}[!t]
  \centering
  \includegraphics[height=3.5ex]{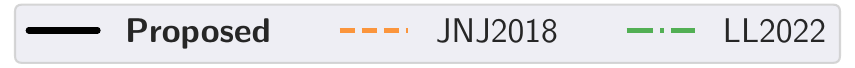}\par%
  \subfloat[Classification\label{fig:exp_classification}]{%
    \includegraphics[width=0.38\linewidth]{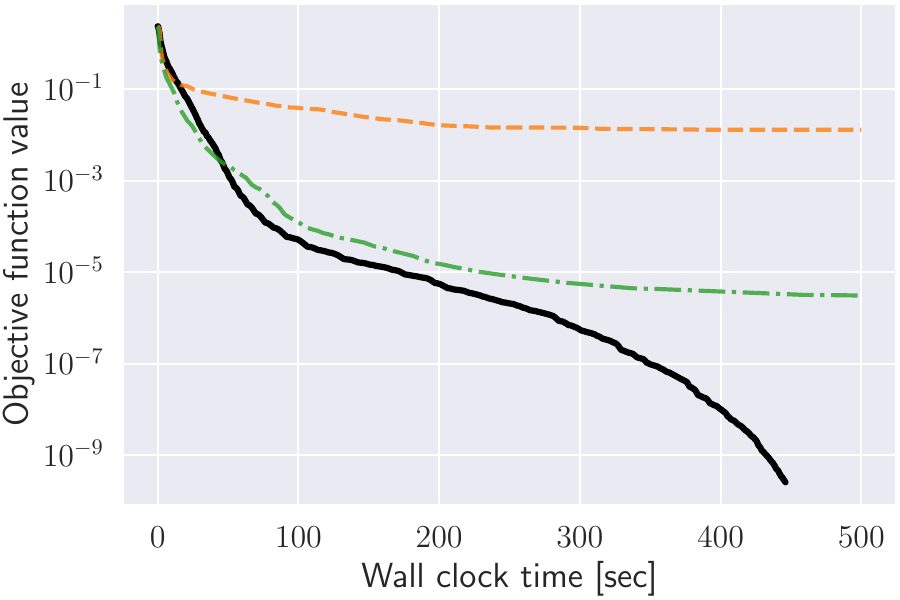}\hspace{0.1\linewidth}%
    \includegraphics[width=0.38\linewidth]{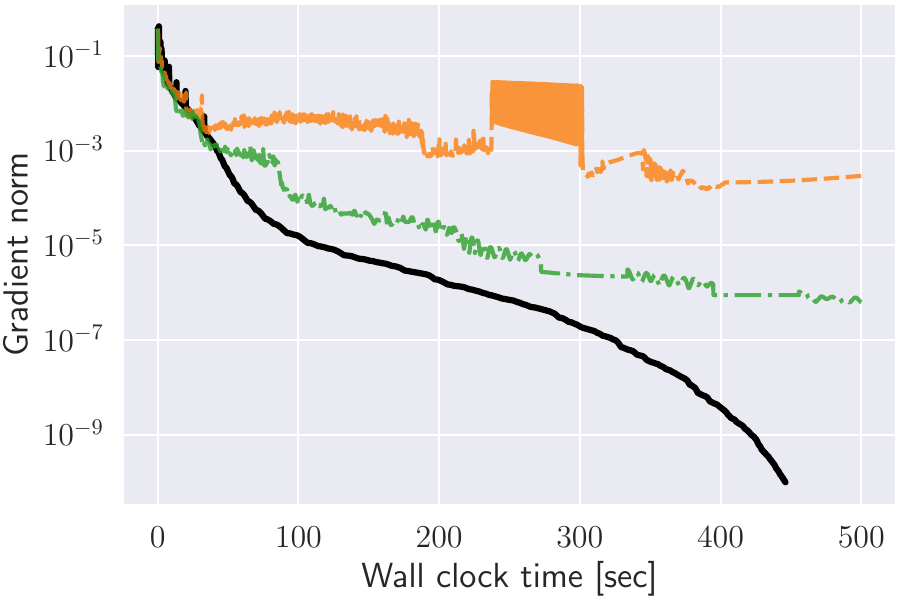}%
  }\par%
  \subfloat[Autoencoder\label{fig:exp_autoencoder}]{%
    \includegraphics[width=0.38\linewidth]{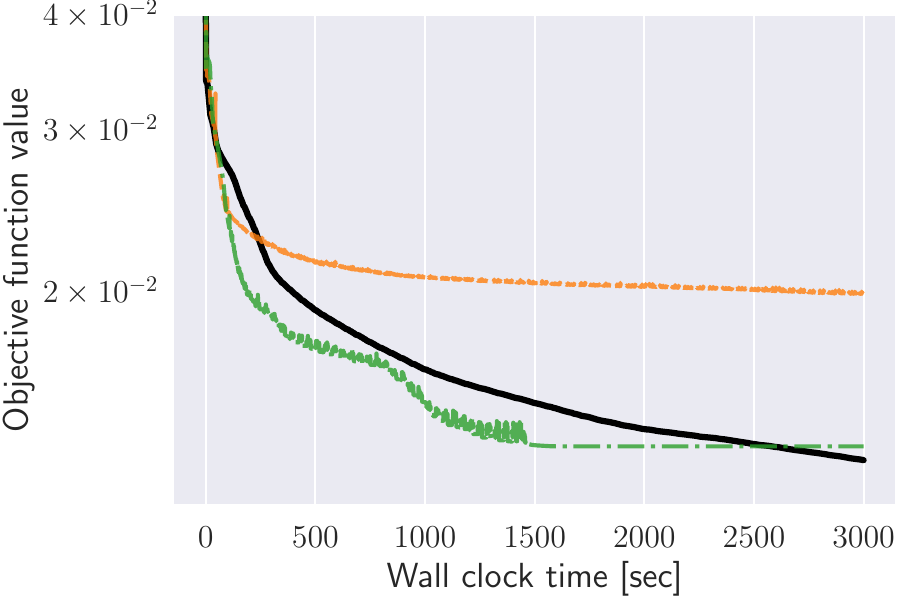}\hspace{0.1\linewidth}%
    \includegraphics[width=0.38\linewidth]{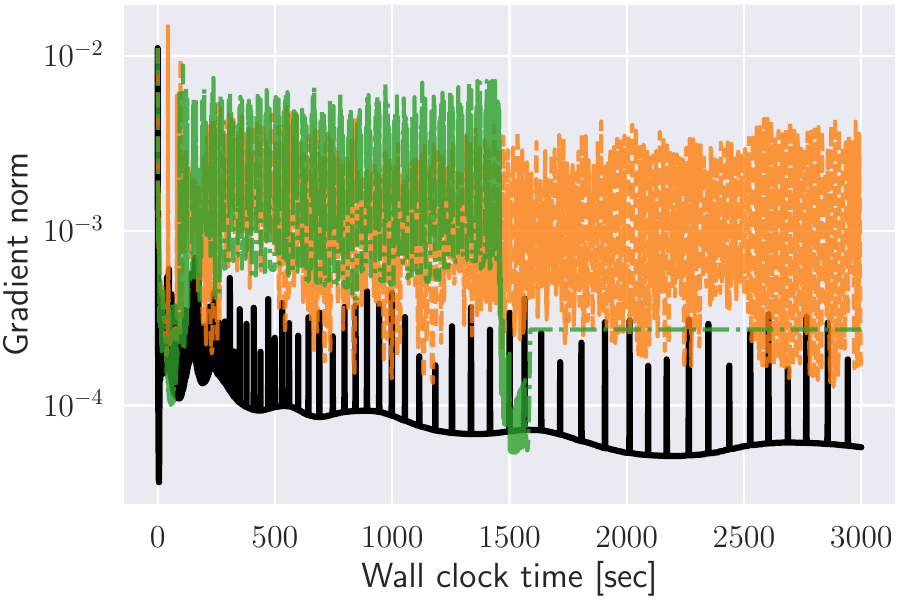}%
  }\par%
  \subfloat[Matrix completion ($r = 100$)\label{fig:exp_completion_100}]{%
    \includegraphics[width=0.38\linewidth]{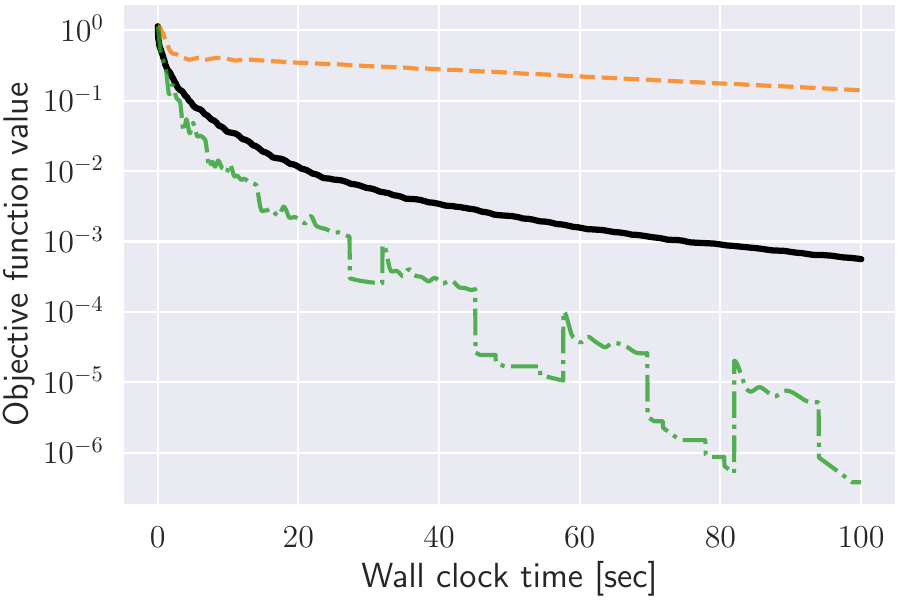}\hspace{0.1\linewidth}%
    \includegraphics[width=0.38\linewidth]{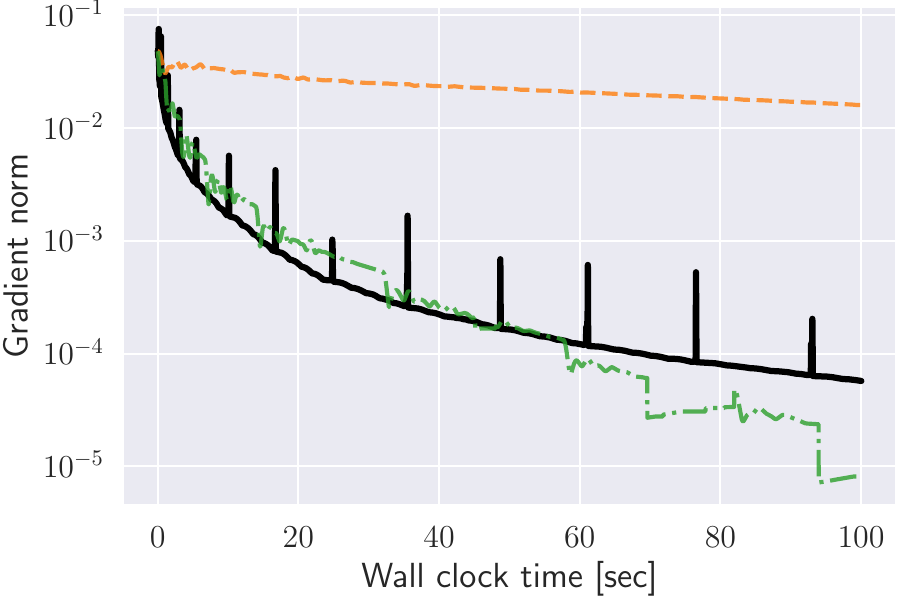}%
  }\par%
  \subfloat[Matrix completion ($r = 200$)\label{fig:exp_completion_200}]{%
    \includegraphics[width=0.38\linewidth]{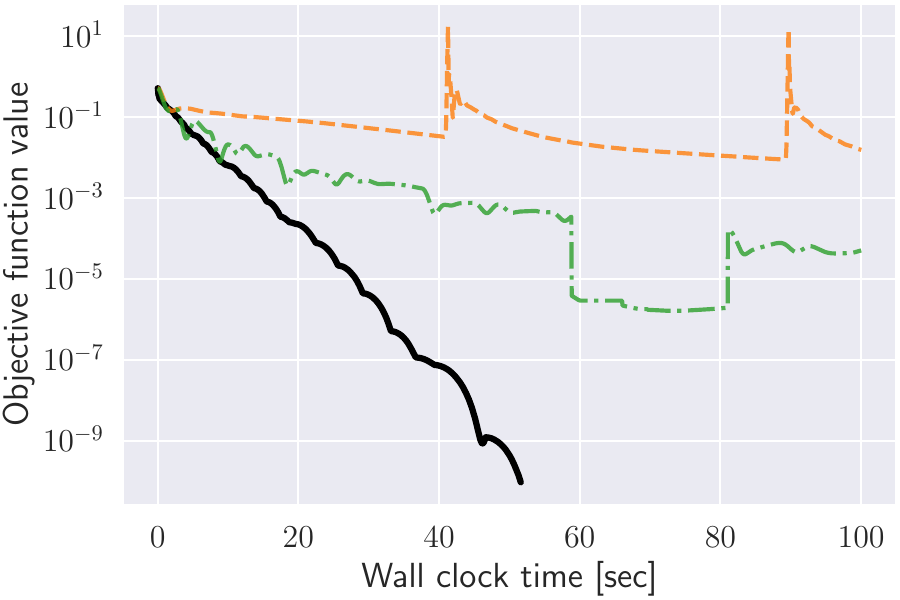}\hspace{0.1\linewidth}%
    \includegraphics[width=0.38\linewidth]{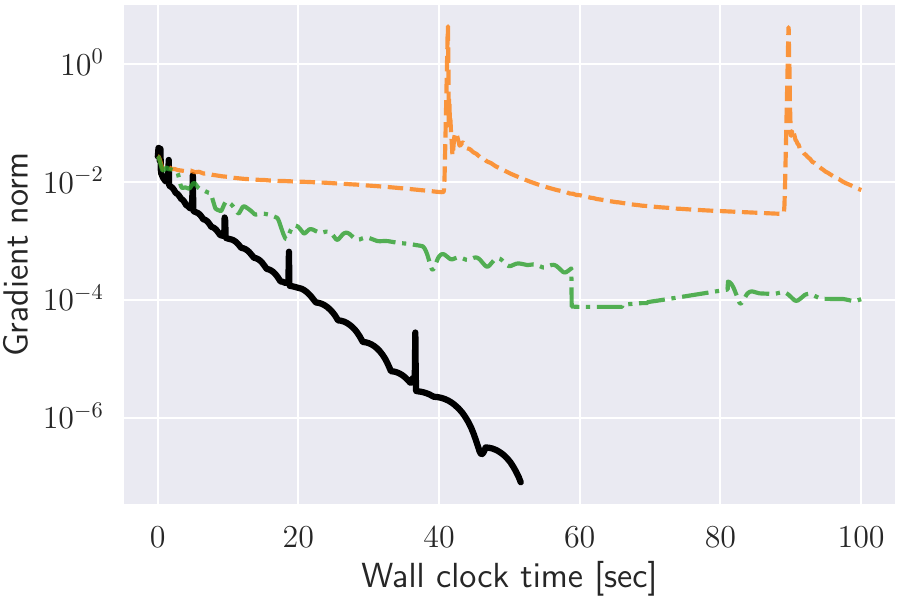}%
  }\par%
  \caption{Numerical comparison with algorithms with $O(\epsilon^{-7/4})$ or $\tilde O(\epsilon^{-7/4})$ complexity bounds.\label{fig:experiments}}
\end{figure}

\subsubsection{Comparison with algorithms without $O(\epsilon^{-7/4})$ or $\tilde O(\epsilon^{-7/4})$ guarantees}
The previous section compares the proposed algorithm with two of the theoretically superior algorithms reviewed in \cref{sec:related_work}.
This section provides further numerical comparisons with \OC, \LBFGS, and \CG, which are known to be practical but have no complexity guarantees or weaker ones than the algorithms in the previous section.
We also provide the results of \GD as a baseline.

The results are given in \cref{fig:experiments_further}, illustrating that \OC, \LBFGS, and \CG perform surprisingly well in practice, though not in theory.\footnote{
  To obtain results of \LBFGS and \CG, we ran the SciPy functions multiple times with the maximum number of iterations set to $2^0, 2^1, 2^2, 2^3,\dots$ because we cannot obtain the solution at each iteration while running SciPy codes of \LBFGS and \CG, but only the final result.
  The results are thus plotted as markers instead of lines in \cref{fig:experiments_further}.
}
Although these methods are often ignored in the current trend of research on algorithms that achieve better complexity bounds, they are so practical that they are worth considering from theoretical perspectives.
In future work, we would like to investigate whether these practical algorithms achieve the same complexity bound as the proposed algorithm or what types of instances can worsen their complexity bounds.



\begin{figure}[!t]
  \centering
  \includegraphics[height=3.5ex]{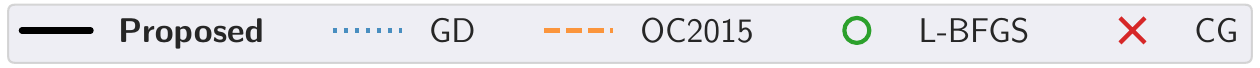}\par%
  \subfloat[Classification\label{fig:exp_further_classification}]{%
    \includegraphics[width=0.38\linewidth]{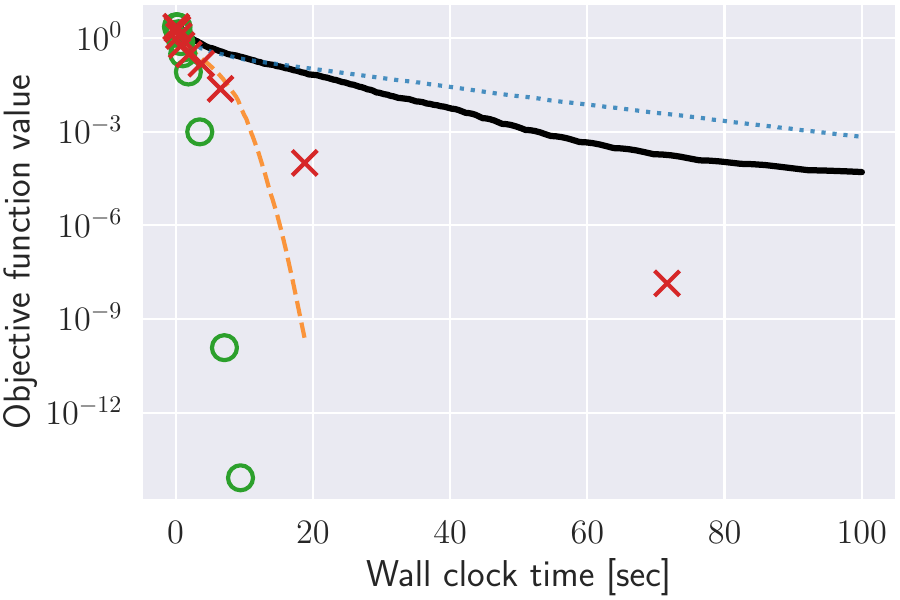}\hspace{0.1\linewidth}%
    \includegraphics[width=0.38\linewidth]{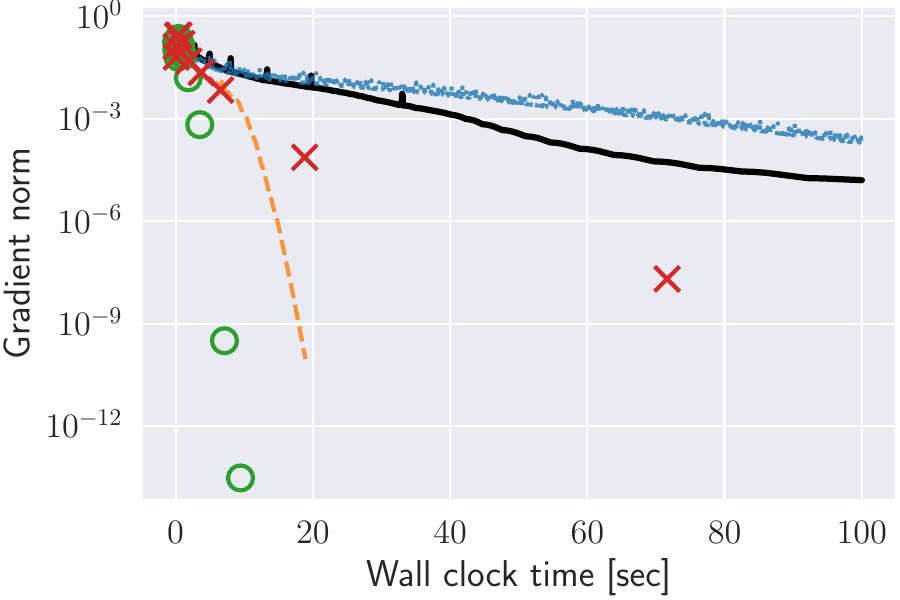}%
  }\par%
  \subfloat[Autoencoder\label{fig:exp_further_autoencoder}]{%
    \includegraphics[width=0.38\linewidth]{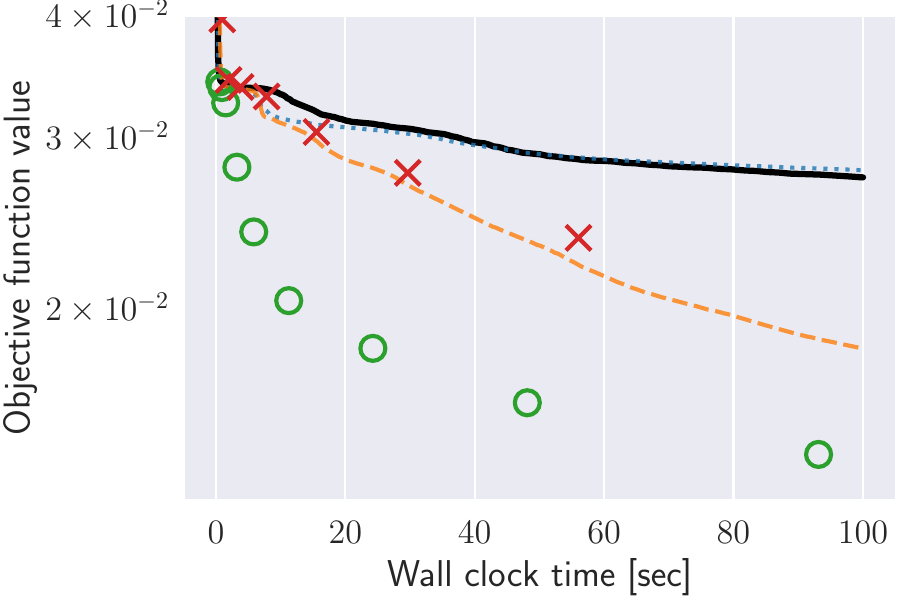}\hspace{0.1\linewidth}%
    \includegraphics[width=0.38\linewidth]{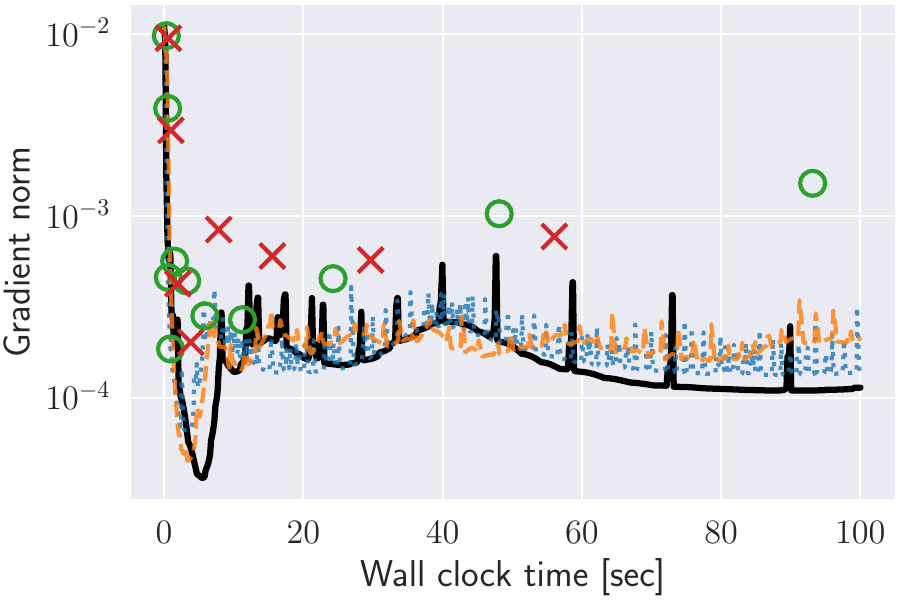}%
  }\par%
  \subfloat[Matrix completion ($r = 100$)\label{fig:exp_further_completion_100}]{%
    \includegraphics[width=0.38\linewidth]{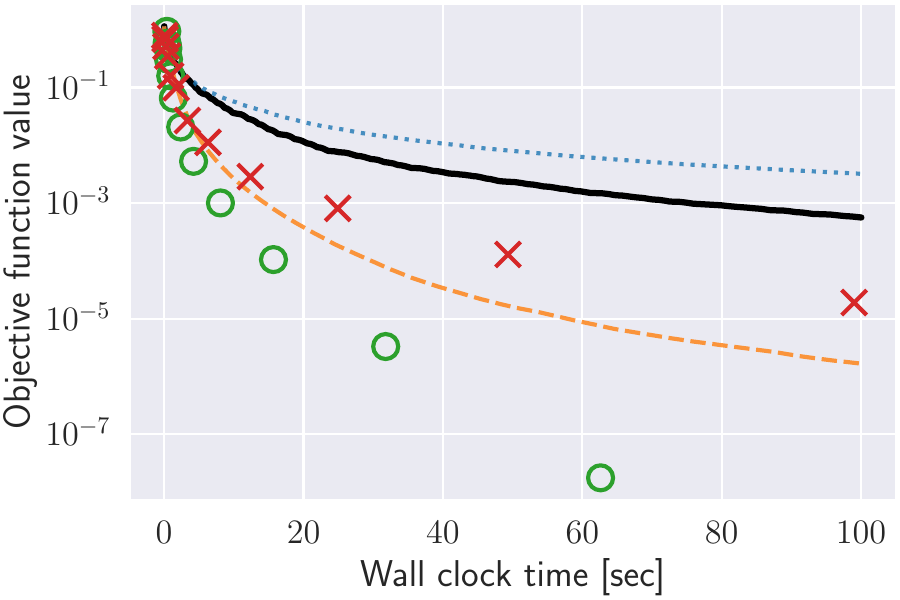}\hspace{0.1\linewidth}%
    \includegraphics[width=0.38\linewidth]{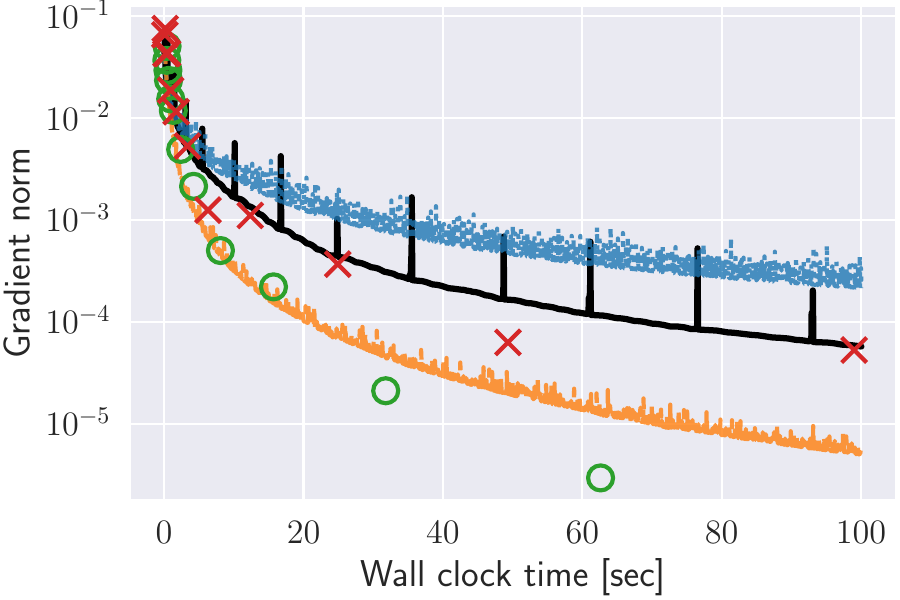}%
  }\par%
  \subfloat[Matrix completion ($r = 200$)\label{fig:exp_further_completion_200}]{%
    \includegraphics[width=0.38\linewidth]{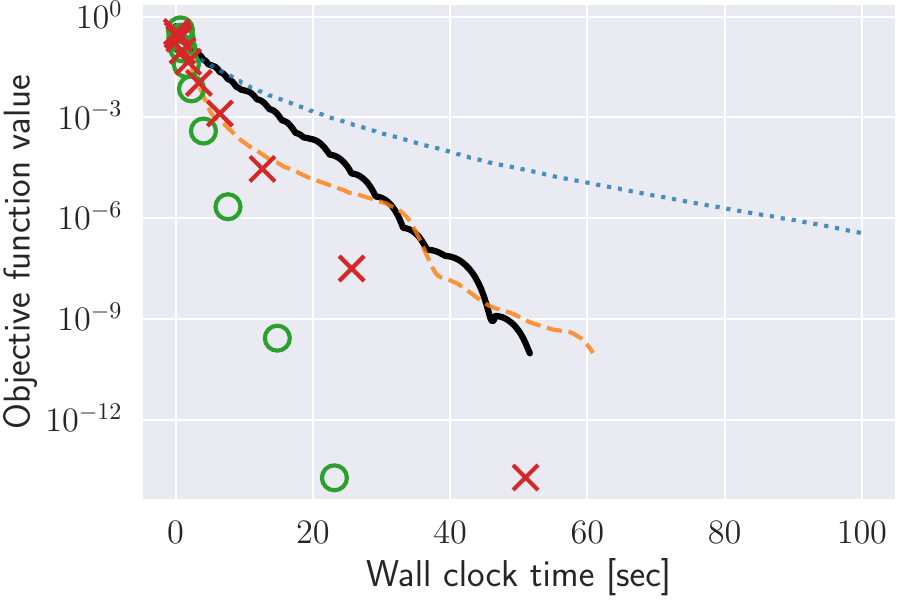}\hspace{0.1\linewidth}%
    \includegraphics[width=0.38\linewidth]{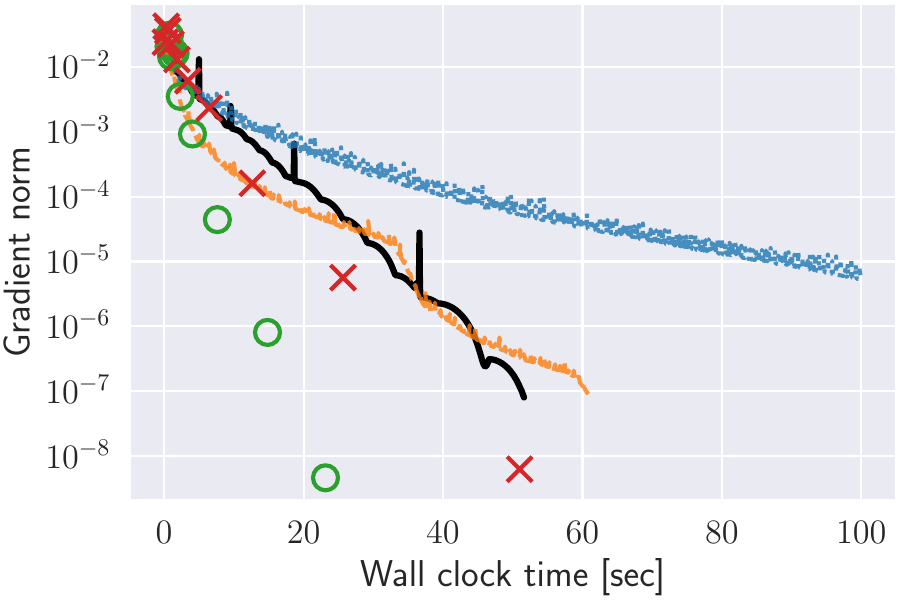}%
  }\par%
  \caption{Numerical comparison with practical algorithms without $O(\epsilon^{-7/4})$ or $\tilde O(\epsilon^{-7/4})$ complexity bounds.\label{fig:experiments_further}}
\end{figure}

\appendix

\section{Proofs for \cref{sec:preliminary}}
\label{sec:proof_lem_lip_hessian}
\begin{proof}[Proof of \cref{lem:gradient_jensen}]
  To simplify the notation, let $\bar z \coloneqq \sum_{i=1}^n \lambda_i z_i$.
  Taylor's theorem gives
  \[
    \nabla f(z_i) - \nabla f(\bar z)
    =
    \nabla^2 f(\bar z) (z_i - \bar z)
    + \int_0^1 \prn*{ \nabla^2 f(\bar z + t (z_i - \bar z)) - \nabla^2 f(\bar z) } (z_i - \bar z) dt,
  \]
  and its weighted average gives
  \begin{align}
    \sum_{i=1}^n \lambda_i \nabla f(z_i) - \nabla f(\bar z)
    &=
    \sum_{i=1}^n \lambda_i
    \int_0^1 \prn*{ \nabla^2 f(\bar z + t (z_i - \bar z)) - \nabla^2 f(\bar z) } (z_i - \bar z) dt.
  \end{align}
  Therefore, we obtain
  \begin{alignat}{2}
    \norm*{
      \sum_{i=1}^n \lambda_i \nabla f(z_i) - \nabla f(\bar z)
    }
    &\leq
    \sum_{i=1}^n \lambda_i
    \int_0^1 \norm*{ \nabla^2 f(\bar z + t (z_i - \bar z)) - \nabla^2 f(\bar z) } \norm*{z_i - \bar z} dt\\
    &\leq
    \sum_{i=1}^n \lambda_i
    \int_0^1 M_f t \norm*{z_i - \bar z}^2 dt
    \quad\by{\cref{asm:hessian_lip}}\\
    &=
    \frac{M_f}{2} \sum_{i=1}^n \lambda_i \norm*{z_i - \bar z}^2.
  \end{alignat}
  Furthermore, elementary algebra shows that $\sum_{i=1}^n \lambda_i \norm*{z_i - \bar z}^2 = \sum_{1 \leq i < j \leq n} \lambda_i \lambda_j \|z_i - z_j\|^2$ as
  \begin{align}
    \sum_{1 \leq i < j \leq n} \lambda_i \lambda_j \norm*{z_i - z_j}^2
    &= 
    \frac{1}{2} \sum_{i,j = 1}^n \lambda_i \lambda_j \norm*{z_i - z_j}^2\\
    &= 
    \frac{1}{2} \sum_{i,j = 1}^n \lambda_i \lambda_j \norm*{z_i}^2
    + \frac{1}{2} \sum_{i,j = 1}^n \lambda_i \lambda_j \norm*{z_j}^2
    - \sum_{i,j = 1}^n \inner{\lambda_i z_i}{\lambda_j z_j}\\
    &= 
    \sum_{i=1}^n \lambda_i \norm*{z_i}^2 - \norm*{\bar z}^2
    =
    \sum_{i=1}^n \lambda_i \norm*{z_i - \bar z}^2,
  \end{align}
  which completes the proof.
\end{proof}

\begin{proof}[Proof of \cref{lem:trapezoidal_rule_error}]
  We obtain the desired result as follows:
  \begin{alignat}{2}
    &\mathInd
      f(x) - f(y)
      - \frac{1}{2} \inner*{\nabla f(x) + \nabla f(y)}{x - y}\\
    &=
      \int_0^1 \inner*{\nabla f(t x + (1 - t) y)}{x - y} dt
      - \frac{1}{2} \inner*{\nabla f(x) + \nabla f(y)}{x - y}
    &\quad&\by{Taylor's\\theorem}\\
    &=
      \int_0^1 \inner[\Big]{\nabla f(t x + (1 - t) y) - t \nabla f(x) - (1 - t) \nabla f(y)}{x - y} dt\\
    &\leq
    \int_0^1 \norm[\Big]{\nabla f(t x + (1 - t) y) - t \nabla f(x) - (1 - t) \nabla f(y)} \norm{x - y} dt\\
    &\leq
    \frac{\Mtrue}{2} \int_0^1 t (1 - t) \norm{x - y}^3 dt
    =
    \frac{\Mtrue}{12} \norm{x - y}^3
    &\quad&\by{\cref{lem:gradient_jensen}}.
  \end{alignat}
\end{proof}

\section{Details of neural networks used in experiments}
\label{sec:detail_neural_net}
For problem~\cref{eq:instance_classification}, we used a three-layer fully connected network with bias parameters.
The layers each had $784$, $32$, $16$, and $10$ nodes and had the logistic sigmoid activation.
The total number of the parameters is $d = (784 \times 32 + 32 \times 16 + 16 \times 10) + (32 + 16 + 10) = 25818$.

For problem~\cref{eq:instance_autoencoder}, we used a four-layer fully connected network with bias parameters.
The layers each had $784$, $32$, $16$, $32$, and $784$ nodes and had the logistic sigmoid activation.
The total number of the parameters is $d = (784 \times 32 + 32 \times 16 + 16 \times 32 + 32 \times 784) + (32 + 16 + 32 + 784) = 52064$.

We initialize the parameters of the neural networks with \texttt{flax.linen.Module.init} method.
See the documentation\footnote{\url{https://flax.readthedocs.io/en/latest/api_reference/flax.linen/module.html}} for more details of the initialization.

\section*{Acknowledgments}
We are deeply grateful to the anonymous reviewers, who carefully read the manuscript and provided helpful comments.
We also thank the associate editor for sharing information on relevant literature.

\bibliographystyle{siamplain}
\bibliography{../myrefs_siam}
\end{document}